\documentclass[12pt]{article}
\usepackage{amsmath,amsfonts,amsthm,amssymb,mathrsfs,mathtools}
\usepackage{graphicx}
\usepackage[usenames,dvipsnames]{color}
\usepackage{float}
\usepackage{graphicx}
\usepackage{subfigure}
\usepackage{caption}
\usepackage{blindtext}
\usepackage{cite}
\usepackage{url}
\usepackage{caption}
\usepackage{epstopdf}
\usepackage{hyperref}

\usepackage{color}
\usepackage[latin1]{inputenc}

\newcommand{\dint}{\displaystyle\int}

\theoremstyle{plain}
\newtheorem{theorem}{Theorem}[section]

\newtheorem{corollary}[theorem]{Corollary}
\newtheorem{lemma}[theorem]{Lemma}
\newtheorem{proposition}[theorem]{Proposition}

\theoremstyle{definition}

\theoremstyle{remark}
\newtheorem{remark}[theorem]{Remark}
\newtheorem{example}[theorem]{Example}
\numberwithin{equation}{section}
\numberwithin{theorem}{section}

\usepackage{a4wide}
\usepackage{geometry}
\begin{document}
	\renewcommand{\thefootnote}{\fnsymbol{footnote}}
	\begin{center}
		{\Large \textbf{Stopping Times in the Filtration of a Brownian Motion Stopped at its Last Passage Time
		}} \\[0pt]
		~\\[0pt] \textbf{Mohammed Louriki}\footnote[1]{Mathematics Department, Faculty of Sciences Semalalia, Cadi Ayyad University, Boulevard Prince Moulay Abdellah,	P. O. Box 2390, Marrakesh 40000, Morocco. E-mail: \texttt{m.louriki@uca.ac.ma}}
		\\[0pt]
	\end{center}
	\begin{abstract}

		We investigate the structural properties of the last passage time $\sigma_z^{\lambda}$ at level $z > 0$ of a Brownian motion with positive drift $\lambda > 0$, denoted $B^{\lambda} = (B_t + \lambda t)_{t \geq 0}$, in the filtration generated by the process $\xi^{\lambda,z} = (B^{\lambda}_{t \wedge \sigma_z^{\lambda}})_{t \geq 0}$. We compute the compensator of $\sigma_z^{\lambda}$ and establish that it is the unique totally inaccessible stopping time in the filtration of $\xi^{\lambda,z}$. Moreover, we provide a canonical decomposition of arbitrary stopping times: for any stopping time $T$, the restriction of $T$ to the set $\{T = \sigma_z^\lambda\}$ is totally inaccessible, while its restriction to $\{T \neq \sigma_z^\lambda\}$ is predictable.
		
		Although the paths of $\xi^{\lambda,z}$ are continuous, the process fails to satisfy the Feller property and is not strong Markov. Nevertheless, we show that its natural filtration is quasi-left-continuous. To overcome these limitations, we consider the extended process $\zeta^{\lambda,z} = (\mathbb{I}_{\{t < \sigma_z^\lambda\}}, \xi_t^{\lambda,z})_{t \geq 0}$, and prove that it is a Feller process. We compute its infinitesimal generator, which allows us to characterize the associated class of martingales and identify the solutions to certain partial differential equations.
	\end{abstract}
	\section{Introduction}
	
	\quad\,\, The study of last passage times (or last exit times) for diffusion processes has a long history in probability theory, with foundational works by Doob \cite{Doob}, Nagasawa \cite{Nagasawa}, Kunita and Watanabe \cite{KW}, Salminen \cite{Salminen1984}, Rogers and Williams \cite{RW}, Chung and Walsh \cite{CW}, and Revuz and Yor \cite{RY}, among others. These studies connect last passage times to key concepts such as transience, recurrence, Doob's $h$-transform, time reversal, and Martin boundary theory. For linear diffusions, Salminen \cite{Salminen1984} expressed the distribution of the last passage time via the transition density and the Green function (see also \cite{S}, \cite{BS} and \cite{EK}). Egami and Kevkhishvili \cite{EK2020} later extended this approach using time-reversal techniques, offering further probabilistic insights. A key contribution of \cite{EK2020} is a decomposition of the diffusion path into components corresponding to regions above and below a fixed threshold. Getoor and Sharpe \cite{GS} analyzed last passage times and their joint distribution with the process location, while Pitman and Yor \cite{PY} obtained density formulas using scale functions for regular diffusions on $\mathbb{R}_+$.
	
	Applications of last passage times in financial modeling are diverse and well documented. A comprehensive overview is provided by Nikeghbali and Platen \cite{NP}, who highlight their relevance in modeling default risk, insider trading, and option pricing. For instance, Elliott et al. \cite{EJY} and Jeanblanc and Rutkowski \cite{JR} studied the pricing of defaultable claims whose payoffs depend on the last time a firm's value falls below a critical level. Related developments appear in \cite{CN} and in Chapters 4 and 5 of \cite{JYC}. Egami and Kevkhishvili \cite{EK2020} proposed a new risk management framework based on the last passage time of a firm's leverage ratio to a specified warning threshold, emphasizing the relevance of the time between this event and actual default. To capture information asymmetry, Imkeller \cite{I} considered a setting where the last passage time of a Brownian motion driving a stock price process is not a stopping time for a regular trader, but becomes a stopping time for an insider through progressive enlargement of the filtration. This setting illustrates how access to additional information via last passage times can generate arbitrage opportunities. Finally, last passage times have also been applied to the pricing of European put and call options. Profeta et al. \cite{PRY} showed that option prices can be expressed in terms of the distribution of last passage times. See also Cheridito et al. \cite{CNP} for related developments.
	
	Motivated by these developments, the present paper adopts a structural and filtration-theoretic perspective on the last passage time $\sigma_z^{\lambda}$ of a Brownian motion with positive drift. We investigate its compensator and accessibility properties within the filtration generated by the stopped process $\xi^{\lambda,z} = (B^{\lambda}_{t \wedge \sigma_z^{\lambda}})_{t \geq 0}$. We prove that $\sigma_z^{\lambda}$ is the \emph{only} totally inaccessible stopping time in this filtration and provide a canonical decomposition of arbitrary stopping times. Although $\xi^{\lambda,z}$ has continuous paths, it fails to satisfy the Feller property and is not strong Markov. Nevertheless, we show that its natural filtration is quasi-left-continuous.
	
	From a modeling perspective, the process $\xi^{\lambda,z}$ may represent a firm's observed risk level, with default or liquidation occurring at the last time the firm reaches a critical threshold $z$. After this point the process stops: the firm exits the market or is no longer monitored. The key feature is that this terminal event is \emph{endogenous} and constitutes the only totally inaccessible event from the market's perspective. This provides an alternative to reduced-form models in which default is introduced exogenously. In this sense, our construction yields a structurally simple yet non-trivial framework in which a unique totally inaccessible time arises endogenously.
	It is important to emphasize that our framework differs from the classical approach based on the progressive enlargement of filtrations (see, e.g., \cite{J, JY1, JY2, AJ, AJR}), where one starts from a filtration $\mathbb{F}$ and enlarges it so that a given random time becomes a stopping time. In that setting, last passage times of diffusions belong to the class of honest times and their structural properties have been studied. In particular, it is proved (see Jeulin \cite[p.~65]{J}) that if $\mathbb{G}$ denotes the progressive enlargement of $\mathbb{F}$ by a random time $\tau$, then $\tau$ is $\mathbb{G}$-totally inaccessible if and only if it avoids $\mathbb{F}$-stopping times.
		In contrast, in the present work we construct from the outset the filtration generated by the stopped process
		\[
		\xi_t^{\lambda,z} = B_{t \wedge \sigma_z^\lambda}^{\lambda},
		\qquad 
		\mathcal{F}_t^{\xi^{\lambda,z}} = \sigma(\xi_s^{\lambda,z},\,0\le s\le t).
		\]
		This filtration contains only the information carried by the stopped diffusion and does not observe the Brownian evolution after the last passage time. In particular, it differs from the progressive enlargement filtration
		\[
		\mathcal{F}_t^B \vee \sigma(\sigma_z^\lambda \wedge t).
		\]
		Indeed, on the event $\{\sigma_z^\lambda < t\}$ we have $\xi_t^{\lambda,z}=z$, while $B_t^\lambda$ continues to evolve randomly, showing that the two filtrations cannot coincide.
		Within this smaller filtration we prove that the last passage time $\sigma_z^\lambda$ is the unique totally inaccessible stopping time.
	
	We proceed as follows. Fix $\lambda > 0$ and $z > 0$. Let $B^\lambda$ be a Brownian motion with drift $\lambda$, and define its last passage time at level $z$ by $\sigma_z^{\lambda} := \sup\{t \geq 0 : B^\lambda_t = z\}$. We define the process $\xi^{\lambda,z}$ as the Brownian motion with drift $\lambda$ stopped at this time: $\xi^{\lambda,z}_t = B^\lambda_{t \wedge \sigma_z^{\lambda}}$. The time $\sigma_z^{\lambda}$ is thus a stopping time with respect to the completed filtration $\mathbb{F}^{\xi^{\lambda,z}, c}$. Our main objective is to characterize this stopping time and study the structure of the filtration it generates.
	
	A notable feature of our model is that the process $\xi^{\lambda,z}$ is not a Feller process and does not satisfy the strong Markov property. As a result, classical results concerning stopping times in Feller filtrations such as the quasi-left-continuity of the filtration and the characterization of stopping times based on the continuity or discontinuity of the process at the stopping time do not directly apply; see Theorem \ref{thmappendixpredictable}. To study the nature of $\sigma_z^{\lambda}$, we adopt a compensator-based approach. It is well known that a stopping time is predictable if and only if its compensator is piecewise constant, accessible if and only if its compensator is purely discontinuous, and totally inaccessible if and only if its compensator is continuous; see Theorem \ref{thmappendixcompensator} and also \cite{BBEcompensator}, where the authors use this approach to analyze the default time in an information-based framework.

	We begin by establishing the semimartingale decomposition of $\xi^{\lambda,z}$ in its own filtration. We prove that $\mathbb{F}^{\xi^{\lambda,z}}$ is right-continuous and derive an explicit decomposition involving a Brownian motion stopped at $\sigma_z^{\lambda}$. For $t < \sigma_z^{\lambda}$, the process $z - \xi^{\lambda,z}_t$ follows the dynamics of a bang-bang Brownian motion see \cite{She}, which arises in P. L\'evy's distributional properties to the case of a Brownian motion with drift see \cite{GShiryaev}. Using Tanaka's formula and the occupation time formula, we compute the local time $L^{\lambda,z}(t, x)$ of $\xi^{\lambda,z}$ and express the compensator $A^{\lambda,z}_t$ of $\sigma_z^{\lambda}$ in terms of this local time. We show that this compensator is continuous, which implies that $\sigma_z^{\lambda}$ is a totally inaccessible stopping time. Moreover, the accumulated local time at level $z$, $L^{\lambda,z}(\sigma_z^{\lambda}, z)$, follows an exponential distribution with parameter $\lambda$.
	
	To analyze the quasi-left-continuity of the filtration \(\mathbb{F}^{\xi^{\lambda,z}, c}\), we construct an augmented process
	$\zeta^{\lambda,z}_t := ( \mathbb{I}_{\{t < \sigma_z^{\lambda}\}},\, \xi^{\lambda,z}_t )$.
	We prove that \(\zeta^{\lambda,z}\) is a Feller process and generates the same completed filtration as \(\xi^{\lambda,z}\). As a consequence, the filtration \(\mathbb{F}^{\xi^{\lambda,z}, c}\) is quasi-left-continuous, and \(\sigma_z^{\lambda}\) is the only totally inaccessible stopping time in this filtration. Moreover, any \(\mathbb{F}^{\xi^{\lambda,z}, c}\)-stopping time \(T\) admits the decomposition
	\[
	T = T_{\{T = \sigma_z^{\lambda}\}} \wedge T_{\{T \neq \sigma_z^{\lambda}\}},
	\]
	where \(T_{\{T = \sigma_z^{\lambda}\}}\) is totally inaccessible and \(T_{\{T \neq \sigma_z^{\lambda}\}}\) is predictable. This highlights the unique structural role of \(\sigma_z^{\lambda}\) as an endogenous, totally inaccessible time within the filtration.
	
	Finally, we compute the infinitesimal generator of \(\zeta^{\lambda,z}\), which characterizes the associated class of martingales and provides a probabilistic representation of the solutions to a family of partial differential equations.

	The paper is structured as follows: Section~\ref{SectionStrongMarkov} examines the Markov, strong Markov, and Feller properties of the process $\xi^{\lambda,z}$. Section \ref{Sectionsemimartingale} focuses on its semi-martingale property. In Section~\ref{Sectioncompensator}, we study the local time of $\xi^{\lambda,z}$ and the compensator of $\sigma_z^{\lambda}$. Section \ref{Sectionaccessiblepredictabletimes} deals with the quasi-left-continuity of the completed filtration generated by $\xi^{\lambda,z}$ and the nature of stopping times within this filtration. Section \ref{sectiongenerator} is devoted to the computation of the generator associated to $\zeta^{\lambda,z}$, and provide the solutions to certain partial differential equations. In Appendix~\ref{Appendix_Foundation}, we provide several preliminary results that are essential for establishing the main findings of this paper. Appendix~\ref{Appendix_section_Proof} contains the proofs of certain auxiliary results. In Appendix~\ref{Appendix_integrals}, we present the detailed computations of the integrals used in the proofs of the main results. Finally, in Appendix~\ref{Appendix_Kallenberg}, for ease of reference, we include a brief overview of the results from~\cite{K} that characterize stopping times using martingales, compensators, and Feller processes.
	
	The following notation will be used throughout the paper: For a complete probability space $(\Omega,\mathcal{F},\mathbb{P})$, $\mathcal{N}_{\mathbb{P}}$ denotes the
	collection of $\mathbb{P}$-null sets. If $\theta$ is a random variable, then $\mathbb{P}_{\theta}$ and $F_{\theta}$ are its law and its distribution function under $\mathbb{P}$, respectively. 
	If $E$ is a topological space, then the Borel $\sigma$-algebra over $E$ will be denoted by $\mathcal{B}(E)$. The characteristic function of a set $A$ is written $\mathbb{I}_{A}$. 
	The symmetric difference of two sets $A$ and $B$ is denoted by $A\Delta B$. The function $p(t, x, y)$, $x, y\in \mathbb{R}$, $t\in\mathbb{R}_+$, denotes the Gaussian density function with variance $t$ and mean $y$, if $y=0$, for simplicity of notation we write $p(t,x)$ rather than $p(t, x, 0)$.
	Finally for any process $Y=(Y_t,\, t\geq 0)$ on $(\Omega,\mathcal{F},\mathbb{P})$, we define by:
	\begin{enumerate}
		\item[(i)] $\mathbb{F}^{Y}=\bigg(\mathcal{F}^{Y}_t:=\sigma(Y_s, s\leq t),~ t\geq 0\bigg)$ the natural filtration of the process $Y$.
		\item[(ii)] $\mathbb{F}^{Y,c}=\bigg(\mathcal{F}^{Y,c}_t:=\mathcal{F}^{Y}_t\vee \mathcal{N}_{\mathbb{P}},\, t\geq 0\bigg)$ the completed natural filtration of the process $Y$.
		\item[(iii)] $\mathbb{F}^{Y,c}_{+}=\bigg(\mathcal{F}^{Y,c}_{t^{+}}:=\underset{{s>t}}\bigcap\mathcal{F}^{Y,c}_{s}=\mathcal{F}^{Y}_{t^{+}}\vee \mathcal{N}_{\mathbb{P}},\, t\geq 0\bigg)$ the smallest filtration containing $\mathbb{F}^{Y}$ and satisfying the usual hypotheses of right-continuity and completeness.
	\end{enumerate}
	\section{Markov, Strong Markov and Feller Properties}
	\label{SectionStrongMarkov}
		Last passage times of diffusions arise naturally in the study of transient processes and belong to the class of honest times, see for instance \cite{AJ,JYC,J,Profeta}. Their structural properties have been extensively investigated in the literature on enlargement of filtrations; see, for instance, \cite{JYC,PY,PRY}, as well as \cite{J}. In particular, their predictable compensators and associated semimartingale decompositions are known in a fairly general setting; see, e.g., \cite{AJ,PY}. 
		Moreover, Girsanov transformations can often be used to relate their distributions to classical cases; see, e.g., \cite{MY}.\\		
		In the present paper we focus on the case of a Brownian motion with drift. More precisely, let $\lambda>0$, $z>0$, and let $B^{\lambda}=(B_t+\lambda t,t\ge0)$ be a Brownian motion with positive drift $\lambda$. Define
		\[
		\sigma_z^{\lambda}=\sup\{t>0:B_t^{\lambda}=z\},
		\] 
		the last passage time of $B^{\lambda}$ at level $z$. Since the drift $\lambda$ is positive, the last passage time $\sigma_z^{\lambda}$ is almost surely finite. The random time $\sigma_z^{\lambda}$ is not an $\mathbb{F}^{B}$-stopping time. Moreover, its distribution is absolutely continuous with respect to the Lebesgue measure and is given by
		\begin{equation}
			\mathbb{P}(\sigma_z^{\lambda}\in \mathrm{d}r)=\lambda\,p(r,z,\lambda r)\,\mathrm{d}r.
			\label{eqsigmazlamdalaw}
		\end{equation}
		The $\mathbb{F}^{B}$-predictable compensator associated with $\sigma_z^{\lambda}$ is given by
		\begin{equation}
			A_t^{\lambda}=\frac{1}{2}\exp(2\lambda z)\,
			L_t^{\exp(-2\lambda z)}(\exp(-2\lambda B^{\lambda})),
		\end{equation}
		where $L^{\exp(-2\lambda z)}(\exp(-2\lambda B^{\lambda}))$ denotes the local time of 
		$\exp(-2\lambda B^{\lambda})$ at level $\exp(-2\lambda z)$ 
		(see \cite{MRY,JYC}). Furthermore, from Propositions 5.6.2.1 and 5.6.2.4 in \cite{JYC}, for any positive predictable process $H$, we have
		\begin{equation}
			\mathbb{E}[H_{\sigma_z^{\lambda}}\mid\sigma_z^{\lambda}=r]
			=\mathbb{E}[H_r\mid B_r^{\lambda}=z].
			\label{eqpredictableformulla}
		\end{equation}
		While these properties have been extensively studied in the enlargement-of-filtration framework, our aim is to analyze the random time $\sigma_z^{\lambda}$ within the filtration generated by the Brownian motion with drift $\lambda$ stopped at $\sigma_z^{\lambda}$, namely the process
	\begin{equation}
	\xi^{\lambda,z}=(B^{\lambda}_{t\wedge \sigma_z^{\lambda}}, t\geq 0).
	\end{equation}
	It is clear that if $\sigma_z^{\lambda}\leq t$, then $\xi_t^{\lambda,z}=z$. However, on the event $\{t<\sigma_z^{\lambda}\}$, we have for any Borel set $A$ and for any positive real number $t$,
	\begin{align*}
		\mathbb{P}(t<\sigma_z^{\lambda},\xi^{\lambda,z}_t\in A)&=\displaystyle\int_{t}^{+\infty}\mathbb{P}(t<\sigma_z^{\lambda},B^{\lambda}_t\in A\vert\sigma_z^{\lambda}=r )\,\mathbb{P}_{\sigma_z^{\lambda}}(\hbox{d}r)\\
		&=\lambda\,\displaystyle\int_{t}^{+\infty}\mathbb{P}(B^{\lambda}_t\in A\vert\sigma_z^{\lambda}=r )\,p(r,z,\lambda\,r)\hbox{d}r\\
		&=\lambda\,\displaystyle\int_{A}\displaystyle\int_{t}^{+\infty}\,p(r-t,z-y,\lambda\,(r-t))\hbox{d}r\,p(t,y,\lambda\,t)\hbox{d}y\\
		&=\displaystyle\int_{A}\exp(-\lambda(\vert z-y\vert -(z-y)))\,p(t,y,\lambda\,t)\hbox{d}y,
	\end{align*}
	where in the latter equality we have used the fact that 
	\begin{equation*}
		\displaystyle\int_{0}^{+\infty}\,p(r,x,\lambda\,r)\hbox{d}r=\dfrac{\exp(-\lambda(\vert x\vert-x))}{\lambda},\,\,\text{for all}\,x\in \mathbb{R}.
	\end{equation*}
	Hence, for any  positive real number $t$,
	\begin{equation}
		\mathbb{P}\left(\{\xi_t^{\lambda,z} = z\} \bigtriangleup \{\sigma_z^{\lambda}\leq t\}\right)=0.\label{eqmodification}
	\end{equation}
	Thus, the random time $\sigma_z^{\lambda}$ is an $\mathbb{F}^{\xi^{\lambda,z},c}$-stopping time. In the remainder of this section, we examine the Markov, strong Markov, and Feller properties of the process $\xi^{\lambda,z}$. See, for instance, \cite{BG}, \cite{K}, \cite{KS}, \cite{SP}, \cite{SW1} and \cite{SW2} for the definitions and properties of Markov, strong Markov, and Feller processes. Situations in which the strong Markov property fails have been studied in various contexts; see, for example, Millar \cite{Millar} and the references therein. For general conditions ensuring the strong Markov property, as well as additional counterexamples, we refer to \cite{D, DY, Y}. We now state a key result: the process $\xi^{\lambda,z}$ is Markovian with explicit transition densities, but it fails to satisfy the strong Markov property.
	\begin{theorem}[Markov property and failure of the strong Markov property of $\xi^{\lambda,z}$]
		The process $\xi^{\lambda,z}$ is a time homogeneous Markov process with respect to its natural filtration, with transition densities given by 
		\begin{multline}
			P_t(x,\mathrm{d}y)=\Biggl\{ \Bigg[\mathbb{I}_{\{x=z\}}+\bigg[\Phi\bigg(\lambda\sqrt{t}- \dfrac{\vert z-x \vert}{\sqrt{t}}\bigg)-\exp(2\lambda\,\vert z-x \vert)\,\Phi\bigg(-\lambda\sqrt{t}- \dfrac{\vert z-x \vert}{\sqrt{t}}\bigg)\bigg]\mathbb{I}_{\{x\neq z\}}\Bigg]
			\mathbb{I}_{\{y=z\}}\\
			+\Bigg[p(t,y-x,\lambda\,t)\,\exp(-\lambda\,\alpha(z,y))\exp(\lambda\,\alpha(z,x))\mathbb{I}_{\{x\neq z\}}\Bigg]\mathbb{I}_{\{y\neq z\}}
			\Biggr\}  \varrho(\mathrm{d}y),\label{eqsemigroup}
		\end{multline}
		where $\Phi$ is the standard normal distribution function,
		\begin{equation}
			\alpha(z,x)=\vert z-x \vert-(z-x),
		\end{equation}
		and  $\varrho$ is the measure defined on $\mathcal{B}(\mathbb{R})$ by
		\begin{equation}\label{measure}
			\varrho(dy)=\delta_z(\mathrm{d}y)+\mathrm{d}y.
		\end{equation}
		However, $\xi^{\lambda,z}$ is not a strong Markov process.
		\label{thmMarkov}
	\end{theorem}
	\begin{proof}
		The Markov property of $\xi^{\lambda,z}$, along with the explicit form of its transition densities, is established in Appendix~\ref{Appendix_section_Proof}. We now show that $\xi^{\lambda,z}$ does not satisfy the strong Markov property. Let $T_{\lambda,z}$ be the first time $\xi^{\lambda,z}$ hits $z$, that is,
		\begin{equation}
			T_{\lambda,z}:=\inf\{t>0,\,\,\xi^{\lambda,z}_t=z\}=\inf\{t>0,\,\,B_{t\wedge \sigma_z^{\lambda}}^{\lambda}=z\}=\inf\{t>0,\,\,B_{t}^{\lambda}=z\}.
		\end{equation}
		Consider the function $f=\mathbb{I}_{\{\mathbb{R}\setminus\{z\}\}}$. Then,
		\begin{equation}
			\mathbb{I}_{\{T_{\lambda,z}<+\infty\}}\mathbb{E}_{\xi^{\lambda,z}_{T_{\lambda,z}}}[f(\xi^{\lambda,z}_{1})]=\mathbb{E}_{z}[f(\xi^{\lambda,z}_{1})]=P_1f(z)=0.\label{eqT_zstrongmarkov}
		\end{equation}
		Choose $D=\mathbb{I}_{\{T_{\lambda,z} \leq 1\}}$. Then, $D \in \mathcal{F}^{\xi^{\lambda,z}}_{T_{\lambda,z}}$ and $\{T_{\lambda,z}<\infty\}$ on $D$. By \eqref{eqT_zstrongmarkov}, $$\mathbb{E}[\mathbb{I}_{D} \mathbb{E}_{\xi^{\lambda,z}_{T_{\lambda,z}}} [f(\xi^{\lambda,z}_1)]]=0.$$ However,
		
		$$
		\begin{aligned}
			\mathbb{E}[\mathbb{I}_{D} f(\xi^{\lambda,z}_{T_{\lambda,z}+1})] & =\mathbb{E}[\mathbb{I}_{\{T_{\lambda,z} \leq 1\}} f(\xi^{\lambda,z}_{T_{\lambda,z}+1})\mathbb{I}_{\{T_{\lambda,z}+1<\sigma_z^{\lambda}\}}]+\mathbb{E}[\mathbb{I}_{\{T_{\lambda,z} \leq 1\}} f(\xi^{\lambda,z}_{T_{\lambda,z}+1})\mathbb{I}_{\{T_{\lambda,z}+1\geq \sigma_z^{\lambda}\}}] \\
			& =\mathbb{E}[\mathbb{I}_{\{T_{\lambda,z} \leq 1\}} f(B^{\lambda}_{T_{\lambda,z}+1})\mathbb{I}_{\{T_{\lambda,z}+1<\sigma_z^{\lambda}\}}]+\mathbb{E}[\mathbb{I}_{\{T_{\lambda,z} \leq 1\}} f(z)\mathbb{I}_{\{T_{\lambda,z}+1\geq \sigma_z^{\lambda}\}}] \\
			& =\int_{0}^{+\infty}\mathbb{E}[\mathbb{I}_{\{T_{\lambda,z} \leq 1\}} f(B^{\lambda}_{T_{\lambda,z}+1})\mathbb{I}_{\{T_{\lambda,z}+1<\sigma_z^{\lambda}\}}\vert\sigma_z^{\lambda}=r ]\mathbb{P}_{\sigma_z^{\lambda}}(\mathrm{d}r)\\
			& =\int_{0}^{+\infty}\mathbb{E}[\mathbb{I}_{\{T_{\lambda,z} \leq 1\}} f(B^{\lambda}_{T_{\lambda,z}+1})\mathbb{I}_{\{T_{\lambda,z}+1<r\}}\vert B^{\lambda}_{r}=z]\mathbb{P}_{\sigma_z^{\lambda}}(\mathrm{d}r)\\
			& =\int_{0}^{+\infty}\mathbb{E}\left[\mathbb{I}_{\{T_{\lambda,z}^{\beta^{r,z}} \leq 1\}} f\left(\beta^{r,z}_{T_{\lambda,z}^{\beta^{r,z}}+1}\right)\mathbb{I}_{\{T_{\lambda,z}^{\beta^{r,z}}+1<r\}}\right]\mathbb{P}_{\sigma_z^{\lambda}}(\mathrm{d}r),
		\end{aligned}
		$$
		where, $\beta^{r,z}$ is a Brownian bridge between $0$ and $z$ of length $r$ and $T_{\lambda,z}^{\beta^{r,z}}$ is the first time $\beta^{r,z}$ hits $z$ before $r$. Using the fact that $\beta^{r,z}$ is a strong Markov process.
		Hence,
		$$
		\begin{aligned}
			\mathbb{E}[\mathbb{I}_{D} f(\xi^{\lambda,z}_{T_{\lambda,z}+1})] & \geq \int_{3}^{+\infty}\mathbb{E}\left[\mathbb{I}_{\{T_{\lambda,z}^{\beta^{r,z}} \leq 1\}} f\left(\beta^{r,z}_{T_{\lambda,z}^{\beta^{r,z}}+1}\right)\right]\mathbb{P}_{\sigma_z^{\lambda}}(\mathrm{d}r)\\
			&= \int_{3}^{+\infty}\mathbb{E}\left[\mathbb{I}_{\{T_{\lambda,z}^{\beta^{r,z}} \leq 1\}}\mathbb{E}\left[ f\left(\beta^{r,z}_{T_{\lambda,z}^{\beta^{r,z}}+1}\right)\bigg\vert \mathcal{F}^{\beta^{r,z}}_{T_{\lambda,z}^{\beta^{r,z}}}\right]\right]\mathbb{P}_{\sigma_z^{\lambda}}(\mathrm{d}r)\\
			&= \int_{3}^{+\infty}\mathbb{E}\left[\mathbb{I}_{\{T_{\lambda,z}^{\beta^{r,z}} \leq 1\}}\mathbb{P}_z\left(\beta^{r,z}_1\neq z \right)\right]\mathbb{P}_{\sigma_z^{\lambda}}(\mathrm{d}r)\\
			&= \int_{3}^{+\infty}\mathbb{P}\left(T_{\lambda,z}^{\beta^{r,z}} \leq 1\right)\mathbb{P}_{\sigma_z^{\lambda}}(\mathrm{d}r)\\
			&= \int_{3}^{+\infty}\mathbb{P}\left(\sup\limits_{s\leq 1}\beta^{r,z}_s\geq z\right)\mathbb{P}_{\sigma_z^{\lambda}}(\mathrm{d}r).
		\end{aligned}
		$$
		It follows from \cite[Theorem 2.1]{BO} that
		\begin{equation}
			\mathbb{P}\left(\sup\limits_{s\leq 1}\beta^{r,z}_s\geq z\right)=2\,\Phi\left(-z\sqrt{\dfrac{r-1}{r}}\right).
		\end{equation}
		Thus,
		\begin{equation*}
			\mathbb{E}[\mathbb{I}_{D} f(\xi^{\lambda,z}_{T_{\lambda,z}+1})]\geq 2\,\Phi\left(-z\right)\left[\Phi\left(-\lambda\sqrt{3}+z/\sqrt{3}\right)+\exp(2\lambda\,z)\Phi\left(-\lambda\sqrt{3}-z/\sqrt{3}\right)\right]>0,
		\end{equation*}
		a contradiction with the strong Markov property. 
	\end{proof}
	\begin{remark}
		\label{remarkfeller}
		Since $\xi^{\lambda,z}$ is not a strong Markov process, it cannot be a Feller process. This is also seen directly by noting that the semigroup $(P_t)$ of $\xi^{\lambda,z}$ does not send $C_0(\mathbb{R})$ into itself: From \eqref{eqsemigroup}, we have 
		\begin{equation*}
			{\displaystyle P_t(f)(x)=\begin{cases}
					&f(z),  \text{if}\,\,x=z,\\
					&f(z)\bigg[ \Phi\bigg(\lambda\sqrt{t}- \dfrac{\vert z-x \vert}{\sqrt{t}}\bigg)-\exp(2\lambda\,\vert z-x \vert)\Phi\bigg(-\lambda\sqrt{t}- \dfrac{\vert z-x \vert}{\sqrt{t}}\bigg)\bigg]+\\
					&\exp(\lambda\alpha(z,x))\displaystyle\int_{\mathbb{R}} f(y)\exp(-\lambda\alpha(z,y))p(t,y-x,\lambda\,t)\hbox{d}y,  \text{if}\,\,x\neq z.
			\end{cases}}\label{eqsemigroup1}
		\end{equation*}
		Let $f(x)=\exp(-\lambda \vert x-z\vert)$ for $x\in \mathbb{R}$. Then, we have
		\begin{multline*}
			\displaystyle\int_{\mathbb{R}}f(y)\exp(-\lambda\alpha(z,y))p(t,y-x,\lambda\,t)\mathrm{d}y\\=\displaystyle\int_{-\infty}^{z}\exp(-\lambda (z-y))p(t,y-x,\lambda\,t)\mathrm{d}y+\displaystyle\int_{z}^{+\infty}\exp(-3\lambda (y-z))p(t,y-x,\lambda\,t)\mathrm{d}y\\
			=\exp\Big(\frac{3}{2}\lambda^2 t+\lambda(x-z)\Big)\bigg[\Phi\bigg(\dfrac{ z-x }{\sqrt{t}}-2\lambda\sqrt{t}\bigg)+\exp\Big(-4\lambda(x-z)\Big) \Phi\bigg(-\dfrac{ z-x }{\sqrt{t}}-2\lambda\sqrt{t}\bigg)\bigg].
		\end{multline*} 
		Hence, for $x\neq z$, we have
		\begin{multline*}
			P_t(f)(x)=\Phi\bigg(\lambda\sqrt{t}- \dfrac{\vert z-x \vert}{\sqrt{t}}\bigg)-\exp\Big(2\lambda\,\vert z-x \vert\Big)\Phi\bigg(-\lambda\sqrt{t}- \dfrac{\vert z-x \vert}{\sqrt{t}}\bigg)+\\
			\exp\Big(\frac{3}{2}\lambda^2 t+\lambda\Big(\vert z-x\vert -2(z-x)\Big)\Big)\bigg[ \Phi\bigg(\dfrac{ z-x }{\sqrt{t}}-2\lambda\sqrt{t}\bigg)\\+\exp\Big(-4\lambda(x-z)\Big) \Phi\bigg(-\dfrac{ z-x }{\sqrt{t}}-2\lambda\sqrt{t}\bigg)\bigg].
		\end{multline*}
		Consequently,
		\begin{equation*}
			\lim\limits_{x\rightarrow z}\,P_t(f)(x)=2\Phi(\lambda\sqrt{t})-1+2\exp\Big(\frac{3}{2}\lambda^2 t\Big)\Phi(-2\lambda\sqrt{t})\neq f(z)=1.
		\end{equation*}
		This shows that the function $x\in \mathbb{R}\longrightarrow P_t(f)(x)$ is not continuous at $z$.
	\end{remark}
Right-continuity of filtrations generated by stopped processes is classical when the stopping time belongs to the underlying filtration; see Chapter~6 of Jeulin \cite{J}. In the present setting, however, the time $\sigma_z^\lambda$ is not a stopping time for the Brownian filtration. In particular, the event $\{\sigma_z^\lambda \le t\}$ does not belong to $\mathcal{F}_t^B$, so the standard arguments do not apply directly.
	\begin{corollary}\label{corfiltrationrightcontinuity}
		The filtration $\mathbb{F}^{\xi^{\lambda,z},c}$ satisfies the usual conditions of right-continuity and completeness. 
	\end{corollary}
	\begin{proof}
		See Appendix \ref{Appendix_section_Proof}.
	\end{proof}
	\section{Semi-martingale Decomposition}
	\label{Sectionsemimartingale}
	In this section, we study the semimartingale decomposition of the process $\xi^{\lambda,z}$. This decomposition plays a crucial role in analyzing both the compensator of the stopping time $\sigma_z^{\lambda}$ and the local time of $\xi^{\lambda,z}$, which will be the focus of the next section. We begin by establishing the canonical decomposition of $\xi^{\lambda,z}$ in its own filtration and derive its explicit representation. This result will provide key insights into the structure of the process and its interaction with stopping times.
	\begin{proposition}
		The process $(B_{t\wedge\sigma_z^{\lambda}}, t\geq 0)$ is not a martingale stopped at $\sigma_z^{\lambda}$ with respect to $\mathbb{F}^{\xi^{\lambda,z}}$. Moreover, we have
		\begin{multline}
			\mathbb{E}[B_{t\wedge\sigma_z^{\lambda}}\vert \mathcal{F}_{s}^{\xi^{\lambda,z}}]=B_{s\wedge\sigma_z^{\lambda}}-\biggl\{2\bigg(\lambda(t-s)-\dfrac{1}{2\lambda}+z-\xi^{\lambda,z}_s\bigg)\exp(2\lambda\,(z-\xi^{\lambda,z}_s))\\
			\times\Phi\bigg(-\lambda\sqrt{t-s}- \dfrac{ z-\xi^{\lambda,z}_s}{\sqrt{t-s}}\bigg)
			+\dfrac{1}{\lambda}\Phi\bigg(\lambda\sqrt{t-s}-\dfrac{ z-\xi^{\lambda,z}_s }{\sqrt{t-s}}\bigg)\\
			-2(t-s)p(t-s,z-\xi^{\lambda,z}_s,\lambda\,(t-s))\biggr\}\mathbb{I}_{\{\xi^{\lambda,z}_s< z\}}\\
			+\biggl\{2\bigg(\lambda(t-s)-\dfrac{1}{2\lambda}+z-\xi^{\lambda,z}_s\bigg)\Phi\bigg(\lambda\sqrt{t-s}+\dfrac{z-\xi^{\lambda,z}_s }{\sqrt{t-s}}\bigg)-2\lambda\,(t-s)\\
			+\dfrac{1}{\lambda}\exp(-2\lambda\,(z-\xi^{\lambda,z}_s))\Phi\bigg(-\lambda\sqrt{t-s}+ \dfrac{ z-\xi^{\lambda,z}_s}{\sqrt{t-s}}\bigg)-2z\\
			+2(t-s)\exp(-2\lambda\, (z-\xi^{\lambda,z}_{s}))p(t-s,z-\xi^{\lambda,z}_s,\lambda\,(t-s))\biggr\}\mathbb{I}_{\{\xi^{\lambda,z}_s> z\}}.\label{eqnotmartingale}
		\end{multline}
		for every $0\leq s\leq t$.
	\end{proposition}
	\begin{proof}
		Let $0\leq s \leq t$, using the Markov property of $\xi^{\lambda,z}$, the fact that $\mathbb{I}_{\{t< \sigma_z^{\lambda} \}}$ is $\sigma(\xi_t^{\lambda,z})\vee \mathcal{N}_{\mathbb{P}}$-measurable, $\sigma_z^{\lambda}\,\mathbb{I}_{\{ s<\sigma_z^{\lambda}\leq t \}}$ is $\sigma(\xi^{\lambda,z}_u, u\geq s)\vee\mathcal{N}_{\mathbb{P}}$-measurable, and that $\sigma_z^{\lambda}\,\mathbb{I}_{\{ \sigma_z^{\lambda}\leq s \}}$ is $\mathcal{F}_{s}^{\xi^{\lambda,z},c}$-measurable, we have
		\begin{multline}
			\mathbb{E}[B_{t\wedge\sigma_z^{\lambda}}\vert \mathcal{F}_{s}^{\xi^{\lambda,z}}]=\mathbb{E}[\xi_t^{\lambda,z}\vert \mathcal{F}_{s}^{\xi^{\lambda,z}}]-\lambda\,\mathbb{E}[t\wedge \sigma_z^{\lambda}\vert \mathcal{F}_{s}^{\xi^{\lambda,z}}]\\
			=\mathbb{E}[\xi_t^{\lambda,z}\vert \xi^{\lambda,z}_s]-\lambda\,t\,\mathbb{P}(t<\sigma_z^{\lambda}\vert \xi^{\lambda,z}_s)-\lambda\,\mathbb{E}[\sigma_z^{\lambda}\,\mathbb{I}_{\{ s<\sigma_z^{\lambda}\leq t \}}\vert \xi^{\lambda,z}_s]-\lambda\,\sigma_z^{\lambda}\,\mathbb{I}_{\{ \sigma_z^{\lambda}\leq s \}}.\label{eqsplit}
		\end{multline}
		From \eqref{eqsemigroup}, we have $\mathbb{P}$-a.s.,
		\begin{multline}
			\mathbb{E}[\xi_t^{\lambda,z}\vert \xi^{\lambda,z}_s]=\bigg[ z\Phi\bigg(\lambda\sqrt{t-s}- \dfrac{\vert z-\xi^{\lambda,z}_{s} \vert}{\sqrt{t-s}}\bigg)\\-z\exp(2\lambda\,\vert z-\xi^{\lambda,z}_{s} \vert)\Phi\bigg(-\lambda\sqrt{t-s}- \dfrac{\vert z-\xi^{\lambda,z}_{s} \vert}{\sqrt{t-s}}\bigg)+\exp(\lambda\,\alpha(z,\xi^{\lambda,z}_{s}))\\
			\times\displaystyle\int_{\mathbb{R}} y\exp(-\lambda\,\alpha(z,y) )p(t-s,y-\xi^{\lambda,z}_{s},\lambda(t-s))\mathrm{d}y\bigg]\,\mathbb{I}_{\{\xi^{\lambda,z}_s\neq z\}}+z\,\mathbb{I}_{\{\xi^{\lambda,z}_s=z\}}.\label{eqconditionalexpectationintegral}
		\end{multline}
		By splitting the integral in \eqref{eqconditionalexpectationintegral} over $(-\infty,z)$ and $(z,+\infty)$ we obtain that, $\mathbb{P}$-a.s.,
		\begin{multline}
			\mathbb{E}[\xi_t^{\lambda,z}\vert \xi^{\lambda,z}_s]=z+\biggl\{ -z\Phi\bigg(-\lambda\sqrt{t-s}+ \dfrac{\vert z-\xi^{\lambda,z}_{s} \vert}{\sqrt{t-s}}\bigg)-z\exp(2\lambda\,\vert z-\xi^{\lambda,z}_{s} \vert)\Phi\bigg(-\lambda\sqrt{t-s}- \dfrac{\vert z-\xi^{\lambda,z}_{s} \vert}{\sqrt{t-s}}\bigg)\\
			+(\xi^{\lambda,z}_{s}+\lambda(t-s))\exp(\lambda\alpha(z,\xi^{\lambda,z}_s))\Phi\bigg(-\lambda\sqrt{t-s}+\dfrac{ z-\xi^{\lambda,z}_{s}}{\sqrt{t-s}}\bigg)\\
			+(\xi^{\lambda,z}_{s}-\lambda(t-s))\exp(\lambda\gamma(z,\xi^{\lambda,z}_s))\Phi\bigg(-\lambda\sqrt{t-s}- \dfrac{ z-\xi^{\lambda,z}_{s}}{\sqrt{t-s}}\bigg)\biggl\}\,\mathbb{I}_{\{\xi^{\lambda,z}_s\neq z\}}.\label{eqE}
		\end{multline}
		where, 
		\begin{equation*}
			\gamma(z,x)=\vert z-x\vert + (z-x).
		\end{equation*}
		Using the identity $\mathbb{I}_{\{ \xi^{\lambda,z}_s\neq z \}}=\mathbb{I}_{\{\xi^{\lambda,z}_s< z\}}+\mathbb{I}_{\{\xi^{\lambda,z}_s> z\}}$,
		we obtain that
		\begin{multline}
			\mathbb{E}[\xi_t^{\lambda,z}\vert\mathcal{F}^{\xi^{\sigma_z^{\lambda}}}_s]=z+\biggr\{\Big(\lambda\,(t-s)-(z-\xi^{\lambda,z}_{s})\Big)\Phi\bigg(-\lambda\sqrt{t-s}+\dfrac{ z-\xi^{\lambda,z}_{s} }{\sqrt{t-s}}\bigg)\\
			-\Big(\lambda\,(t-s)+(z-\xi^{\lambda,z}_{s})\Big)\exp(2\lambda\, (z-\xi^{\lambda,z}_{s}))\Phi\bigg(-\lambda\sqrt{t-s}- \dfrac{ z-\xi^{\lambda,z}_{s}}{\sqrt{t-s}}\bigg)\biggr\}\,\mathbb{I}_{\{\xi^{\lambda,z}_s< z\}}\\
			+\biggr\{\Big(\lambda\,(t-s)-(z-\xi^{\lambda,z}_{s})\Big)\exp(-2\lambda\, (z-\xi^{\lambda,z}_{s}))\Phi\bigg(-\lambda\sqrt{t-s}+ \dfrac{ z-\xi^{\lambda,z}_{s}}{\sqrt{t-s}}\bigg)\\
			-\Big(\lambda\,(t-s)-(z-\xi^{\lambda,z}_{s})\Big)\Phi\bigg(-\lambda\sqrt{t-s}- \dfrac{ z-\xi^{\lambda,z}_{s} }{\sqrt{t-s}}\bigg)\biggr\}\,\mathbb{I}_{\{\xi^{\lambda,z}_s>z\}}.\label{eqEsplit}
		\end{multline}
		On the other hand using \eqref{eqlawofsigmigivenxit} and \eqref{eqintformullanum}, we obtain $\mathbb{P}$-a.s.,
		\begin{multline}
			\mathbb{P}(t<\sigma_z^{\lambda}\vert \xi^{\lambda,z}_s)=\lambda\,\exp(\lambda\alpha(z,\xi^{\lambda,z}_s))\displaystyle\int_t^{+\infty}p(r-s,z-\xi^{\lambda,z}_s,\lambda(r-s))\mathrm{d}r=\\
			\biggl\{\Phi\bigg(-\lambda\sqrt{t-s}+\dfrac{ \vert z-\xi^{\lambda,z}_{s}\vert}{\sqrt{t-s}}\bigg)+\exp(2\lambda\,\vert z-\xi^{\lambda,z}_{s}\vert)\Phi\bigg(-\lambda\sqrt{t-s}-\dfrac{ \vert z-\xi^{\lambda,z}_{s}\vert}{\sqrt{t-s}}\bigg)\biggr\}\mathbb{I}_{\{\xi^{\lambda,z}_s\neq z\}}.\label{eqt<sigmagivenxi_s}
		\end{multline}
		To compute the other term we use once again \eqref{eqlawofsigmigivenxit}, 
		\begin{align*}
			\mathbb{E}[\sigma_z^{\lambda}\,\mathbb{I}_{\{ s<\sigma_z^{\lambda}\leq t \}}\vert \xi^{\lambda,z}_s]&=\lambda\,\exp(\lambda\alpha(z,\xi^{\lambda,z}_s))\displaystyle\int_s^{t}r\,p(r-s,z-\xi^{\lambda,z}_s,\lambda(r-s))\hbox{d}r\mathbb{I}_{\{\xi^{\lambda,z}_s\neq z\}}\\
			&=\lambda\exp(\lambda\alpha(z,\xi^{\lambda,z}_s))\biggl\{s\,\displaystyle\int_s^{t}\,p(r-s,z-\xi^{\lambda,z}_s,\lambda(r-s))\hbox{d}r\\
			&+\displaystyle\int_s^{t}(r-s)\,p(r-s,z-\xi^{\lambda,z}_s,\lambda(r-s))\hbox{d}r\biggr\}\mathbb{I}_{\{\xi^{\lambda,z}_s\neq z\}}, 
		\end{align*}
		Thus, using \eqref{eqintformullanum} and \eqref{eqintformullanum(r-s)}, we obtain
		\begin{multline}
			\mathbb{E}[\sigma_z^{\lambda}\,\mathbb{I}_{\{ s<\sigma_z^{\lambda}\leq t \}}\vert \xi^{\lambda,z}_s]=\biggl\{s\bigg[\Phi\bigg(\lambda\sqrt{t-s}- \dfrac{\vert z-\xi^{\lambda,z}_s \vert}{\sqrt{t-s}}\bigg)-\exp(2\lambda\vert z-\xi^{\lambda,z}_s\vert)\Phi\bigg(-\lambda\sqrt{t-s}- \dfrac{\vert z-\xi^{\lambda,z}_s \vert}{\sqrt{t-s}}\bigg)\bigg]\\
			+\dfrac{1}{\lambda^2}\bigg[\Phi\bigg(\lambda\sqrt{t-s}- \dfrac{\vert z-\xi^{\lambda,z}_s \vert}{\sqrt{t-s}}\bigg)-\exp(2\lambda\vert z-\xi^{\lambda,z}_s\vert)\Phi\bigg(-\lambda\sqrt{t-s}- \dfrac{\vert z-\xi^{\lambda,z}_s \vert}{\sqrt{t-s}}\bigg)\bigg]\\
			+\dfrac{\vert z-\xi^{\lambda,z}_s \vert}{\lambda}\bigg[\Phi\bigg(\lambda\sqrt{t-s}- \dfrac{\vert z-\xi^{\lambda,z}_s \vert}{\sqrt{t-s}}\bigg)+\exp(2\lambda\vert z-\xi^{\lambda,z}_s\vert)\Phi\bigg(-\lambda\sqrt{t-s}- \dfrac{\vert z-\xi^{\lambda,z}_s \vert}{\sqrt{t-s}}\bigg)\bigg]\\
			-2\frac{t-s}{\lambda}\exp(\lambda\alpha(z,\xi^{\lambda,z}_s))p(t-s,z-\xi^{\lambda,z}_s,\lambda\,(t-s))\biggr\}\mathbb{I}_{\{\xi^{\lambda,z}_s\neq z\}}.\label{eqsigmas<sigma<tgivinxi_s}
		\end{multline}
		Inserting \eqref{eqEsplit} and \eqref{eqt<sigmagivenxi_s} into \eqref{eqsigmas<sigma<tgivinxi_s}, we obtain \eqref{eqnotmartingale}, which is the desired result.
	\end{proof}
	In the next theorem, we determine the canonical decomposition of $\xi^{\lambda,z}$ within its own filtration.
	\begin{theorem}\label{thmsemimartingale}
		The process $\xi^{\lambda,z}$ is a semi-martingale and it has the following dynamic:
		\begin{equation}
			\xi^{\lambda,z}_{t}=b^{\lambda,z}_t+\lambda\,\displaystyle\int_{0}^{t\wedge\sigma_z^{\lambda} }\text{sgn}(z-\xi^{\lambda,z}_{s})\,\mathrm{d}s,\, t\geq 0,\label{equationdecompositionsemisigma}
		\end{equation}
		where $b^{\lambda,z}$ is an $\mathbb{F}^{\xi^{\lambda,z}}$-Brownian motion stopped at $\sigma_z^{\lambda}$. Here, the sgn function is defined as
		$$
		\text{sgn}(x):=\left\{\begin{aligned}
			-1, & \text { for } x<0 \\
			0, & \text { for } x=0 \\
			+1, & \text { for } x>0
		\end{aligned}\right..
		$$
	\end{theorem}
	\begin{proof}
		To establish that $b^{\lambda,z}$ is an $\mathbb{F}^{\xi^{\lambda,z}}$-Brownian motion stopped at $\sigma_z^{\lambda}$, we aim to show that it is a martingale and that its quadratic variation is equal to $t\wedge\sigma_z^{\lambda}$. Straightforward estimates allow to check that the integral at the right-hand side of \eqref{equationdecompositionsemisigma} is well defined. The first step of proving Theorem \ref{thmsemimartingale} is to show that
		$$
		\mathbb{E}\left[b^{\lambda,z}_t-b^{\lambda,z}_s \vert\mathcal{F}^{\xi^{\lambda,z}}_s\right]=0, \text { for all } 0 \leq s<t
		$$
		It follows, from \eqref{equationdecompositionsemisigma} and the fact that $\xi^{\lambda,z}$ is Markovian, that
		\begin{align}
			\mathbb{E}[b^{\lambda,z}_t-b^{\lambda,z}_s\vert\mathcal{F}^{\xi^{\lambda,z}}_s]&=\mathbb{E}[\xi^{\lambda,z}_{t}-\xi^{\lambda,z}_{s}\vert\mathcal{F}^{\xi^{\lambda,z}}_s]-\lambda\displaystyle\int_{s}^{t}\mathbb{E}\Big[\text{sgn}(z-\xi^{\lambda,z}_{u})\mathbb{I}_{\{u<\sigma_z^{\lambda}\}}\Big\vert\mathcal{F}^{\xi^{\lambda,z}}_s\Big]\,\mathrm{d}u\nonumber\\
			&=\mathbb{E}[\xi^{\lambda,z}_{t}-\xi^{\lambda,z}_{s}\vert\xi^{\lambda,z}_s]-\lambda\displaystyle\int_{s}^{t}\mathbb{E}\Big[\text{sgn}(z-\xi^{\lambda,z}_{u})\mathbb{I}_{\{u<\sigma_z^{\lambda}\}}\Big\vert\xi^{\lambda,z}_s\Big]\,\mathrm{d}u.\label{eqexpressionneeded}
		\end{align}
		On one hand, using \eqref{eqEsplit} we have
		\begin{multline}
			\mathbb{E}[\xi_t^{\lambda,z}-\xi^{\lambda,z}_s\vert \xi^{\lambda,z}_s]=\biggl\{(\xi^{\lambda,z}_{s}-z+\lambda(t-s))\Phi\bigg(-\lambda\sqrt{t-s}+\dfrac{ z-\xi^{\lambda,z}_{s}}{\sqrt{t-s}}\bigg)+\\
			(\xi^{\lambda,z}_{s}-z-\lambda(t-s))\exp(2\lambda( z-\xi^{\lambda,z}_{s}) )\Phi\bigg(-\lambda\sqrt{t-s}-\dfrac{ z-\xi^{\lambda,z}_{s}}{\sqrt{t-s}}\bigg)+z-\xi^{\lambda,z}_{s}\biggl\}\,\mathbb{I}_{\{\xi^{\lambda,z}_s<z\}}\\
			+\biggl\{(\xi^{\lambda,z}_{s}-z+\lambda(t-s))\exp(-2\lambda( z-\xi^{\lambda,z}_{s}) )\Phi\bigg(-\lambda\sqrt{t-s}+\dfrac{ z-\xi^{\lambda,z}_{s}}{\sqrt{t-s}}\bigg)\\
			+(\xi^{\lambda,z}_{s}-z-\lambda(t-s))\Phi\bigg(-\lambda\sqrt{t-s}-\dfrac{z-\xi^{\lambda,z}_{s}}{\sqrt{t-s}}\bigg)+z-\xi^{\lambda,z}_{s}\biggl\}\,\mathbb{I}_{\{\xi^{\lambda,z}_s>z\}}.\label{eqfirsttermexpression}
		\end{multline}
		On the other hand, we have
		\begin{multline}
			\mathbb{E}\Big[\text{sgn}(z-\xi^{\lambda,z}_{u})\mathbb{I}_{\{u<\sigma_z^{\lambda}\}}\Big\vert\xi^{\lambda,z}_s\Big]=\exp(\lambda\,\alpha(z,\xi^{\lambda,z}_{s}))\biggl\{ \displaystyle\int_{-\infty}^{z} p(u-s,y-\xi^{\lambda,z}_{s},\lambda(u-s))\mathrm{d}y\\
			-  \displaystyle\int_{z}^{+\infty} \exp(-2\lambda\,(y-z) )p(u-s,y-\xi^{\lambda,z}_{s},\lambda(u-s))\mathrm{d}y\biggl\}\,\mathbb{I}_{\{\xi^{\lambda,z}_s\neq z\}}\\
			=\biggl\{ \exp(\lambda\,\alpha(z,\xi^{\lambda,z}_{s}))-  \exp(\lambda\,\gamma(z,\xi^{\lambda,z}_{s}))
			+\exp(\lambda\,\gamma(z,\xi^{\lambda,z}_{s}))\Phi\bigg(\lambda\sqrt{u-s}+ \dfrac{ z-\xi^{\lambda,z}_{s} }{\sqrt{u-s}}\bigg)\\-\exp(\lambda\,\alpha(z,\xi^{\lambda,z}_{s}))\Phi\bigg(\lambda\sqrt{u-s}- \dfrac{ z-\xi^{\lambda,z}_{s} }{\sqrt{u-s}}\bigg)\biggl\}\,\mathbb{I}_{\{\xi^{\lambda,z}_s\neq z\}}.
		\end{multline}
		Thus, by integrating the previous expression from $s$ to $t$, we obtain
		\begin{multline}
			\displaystyle\int_{s}^{t}\mathbb{E}\Big[\text{sgn}(z-\xi^{\lambda,z}_{u})\mathbb{I}_{\{u<\sigma_z^{\lambda}\}}\Big\vert\xi^{\lambda,z}_s\Big]\,\mathrm{d}u=\biggl\{(t-s)\Big(\exp(\lambda\,\alpha(z,\xi^{\lambda,z}_{s}))-  \exp(\lambda\,\gamma(z,\xi^{\lambda,z}_{s}))\Big)\\+\exp(\lambda\,\gamma(z,\xi^{\lambda,z}_{s})) \displaystyle\int_{0}^{t-s}\Phi\bigg(\lambda\sqrt{u}+ \dfrac{ z-\xi^{\lambda,z}_{s} }{\sqrt{u}}\bigg)\mathrm{d}u\\
			-\exp(\lambda\,\alpha(z,\xi^{\lambda,z}_{s}))\displaystyle\int_{0}^{t-s}\Phi\bigg(\lambda\sqrt{u}-\dfrac{ z-\xi^{\lambda,z}_{s} }{\sqrt{u}}\bigg)\mathrm{d}u\biggl\}\,\mathbb{I}_{\{\xi^{\lambda,z}_s\neq z\}}.\label{eqsecondterm}
		\end{multline}
		It follows from \eqref{eqappendixint1} and \eqref{eqappendixint2} that
		\begin{multline}
			\displaystyle\int_{0}^{t-s}\Phi\bigg(\lambda\sqrt{u}+ \dfrac{ z-\xi^{\lambda,z}_{s} }{\sqrt{u}}\bigg)\mathrm{d}u=\frac{1}{2 \lambda^2} \exp(-\lambda\gamma(z,\xi^{\lambda,z}_s))
			\left(-\lambda\alpha(z,\xi^{\lambda,z}_s)-1\right) \Phi\left(\lambda \sqrt{t-s}-\frac{\vert z-\xi^{\lambda,z}_s\vert}{\sqrt{t-s}}\right)\\
			+\frac{1}{2 \lambda^2} \exp(\lambda\alpha(z,\xi^{\lambda,z}_s))\left(\lambda\gamma(z,\xi^{\lambda,z}_s)-1\right)
			\left( \Phi\left(\lambda \sqrt{t-s}+\frac{\vert z-\xi^{\lambda,z}_s\vert}{\sqrt{t-s}}\right)-1\right) \\
			+(t-s) \Phi\left(\lambda \sqrt{t-s}+\frac{ z-\xi^{\lambda,z}_s}{\sqrt{t-s}}\right)
			+\dfrac{\sqrt{t-s}}{\lambda}\dfrac{1}{\sqrt{2\pi}}\exp\bigg(-\frac{1}{2}\Big(\lambda \sqrt{t-s}+\dfrac{z-\xi^{\lambda,z}_s}{\sqrt{t-s}}\Big)^2\bigg),\label{eqint1}
		\end{multline}
		and that
		\begin{multline}
			\displaystyle\int_{0}^{t-s}\Phi\bigg(\lambda\sqrt{u}- \dfrac{ z-\xi^{\lambda,z}_{s} }{\sqrt{u}}\bigg)\mathrm{d}u=\frac{1}{2 \lambda^2} \exp(-\lambda\alpha(z,\xi^{\lambda,z}_s))
			\left(-\lambda\gamma(z,\xi^{\lambda,z}_s)-1\right) \Phi\left(\lambda \sqrt{t-s}-\frac{\vert z-\xi^{\lambda,z}_s\vert}{\sqrt{t-s}}\right)\\
			+\frac{1}{2 \lambda^2} \exp(\lambda\gamma(z,\xi^{\lambda,z}_s))\left(\lambda\alpha(z,\xi^{\lambda,z}_s)-1\right)
			\left( \Phi\left(\lambda \sqrt{t-s}+\frac{\vert z-\xi^{\lambda,z}_s\vert}{\sqrt{t-s}}\right)-1\right)+ \\
			(t-s) \Phi\left(\lambda \sqrt{t-s}-\frac{ z-\xi^{\lambda,z}_s}{\sqrt{t-s}}\right)
			+\dfrac{\sqrt{t-s}}{\lambda}\dfrac{1}{\sqrt{2\pi}}\exp\bigg(-\frac{1}{2}\Big(\lambda \sqrt{t-s}-\dfrac{z-\xi^{\lambda,z}_s}{\sqrt{t-s}}\Big)^2\bigg),\label{eqint2}
		\end{multline}
		Hence, we have
		\begin{multline}
			\exp(\lambda\,\gamma(z,\xi^{\lambda,z}_{s})) \displaystyle\int_{0}^{t-s}\Phi\bigg(\lambda\sqrt{u}+ \dfrac{ z-\xi^{\lambda,z}_{s} }{\sqrt{u}}\bigg)\mathrm{d}u
			-\exp(\lambda\,\alpha(z,\xi^{\lambda,z}_{s}))\displaystyle\int_{0}^{t-s}\Phi\bigg(\lambda\sqrt{u}-\dfrac{ z-\xi^{\lambda,z}_{s} }{\sqrt{u}}\bigg)\mathrm{d}u=\\
			-\Big(\frac{1}{\lambda}(z-\xi^{\lambda,z}_s)-(t-s)\exp(\lambda\,\alpha(z,\xi^{\lambda,z}_{s}))\Big)\Phi\left(-\lambda \sqrt{t-s}+\frac{\vert z-\xi^{\lambda,z}_s\vert}{\sqrt{t-s}}\right)\\
			-\Big(\frac{1}{\lambda}(z-\xi^{\lambda,z}_s)\exp(2\lambda\vert z-\xi^{\lambda,z}_s\vert)+(t-s)\exp(\lambda\,\gamma(z,\xi^{\lambda,z}_{s}))\Big)\Phi\left(-\lambda \sqrt{t-s}-\frac{\vert z-\xi^{\lambda,z}_s\vert}{\sqrt{t-s}}\right)\\
			+\frac{1}{\lambda}(z-\xi^{\lambda,z}_s)+(t-s)\exp(\lambda\,\gamma(z,\xi^{\lambda,z}_{s}))-(t-s)\exp(\lambda\,\alpha(z,\xi^{\lambda,z}_{s})). \label{eqint1int2}
		\end{multline}
		Substituting \eqref{eqint1int2} in \eqref{eqsecondterm} and multiplying by $\lambda$, we obtain
		\begin{multline}
			\lambda\displaystyle\int_{s}^{t}\mathbb{E}\Big[\text{sgn}(z-\xi^{\lambda,z}_{u})\mathbb{I}_{\{u<\sigma_z^{\lambda}\}}\Big\vert\xi^{\lambda,z}_s\Big]\,\mathrm{d}u=\biggl\{(\xi^{\lambda,z}_{s}-z+\lambda(t-s))\Phi\bigg(-\lambda\sqrt{t-s}+\dfrac{ z-\xi^{\lambda,z}_{s}}{\sqrt{t-s}}\bigg)+\\
			(\xi^{\lambda,z}_{s}-z-\lambda(t-s))\exp(2\lambda( z-\xi^{\lambda,z}_{s}) )\Phi\bigg(-\lambda\sqrt{t-s}-\dfrac{ z-\xi^{\lambda,z}_{s}}{\sqrt{t-s}}\bigg)+z-\xi^{\lambda,z}_{s}\biggl\}\,\mathbb{I}_{\{\xi^{\lambda,z}_s<z\}}\\
			+\biggl\{(\xi^{\lambda,z}_{s}-z+\lambda(t-s))\exp(-2\lambda( z-\xi^{\lambda,z}_{s}) )\Phi\bigg(-\lambda\sqrt{t-s}+\dfrac{ z-\xi^{\lambda,z}_{s}}{\sqrt{t-s}}\bigg)\\
			+(\xi^{\lambda,z}_{s}-z-\lambda(t-s))\Phi\bigg(-\lambda\sqrt{t-s}-\dfrac{z-\xi^{\lambda,z}_{s}}{\sqrt{t-s}}\bigg)+z-\xi^{\lambda,z}_{s}\biggl\}\,\mathbb{I}_{\{\xi^{\lambda,z}_s>z\}}.\label{eqsecondtermexpression}
		\end{multline}
		Finally, inserting \eqref{eqfirsttermexpression} and \eqref{eqsecondtermexpression} into \eqref{eqexpressionneeded}, we obtain, for all $s\leq t$,
		\begin{equation}
			\mathbb{E}[b^{\lambda,z}_t-b^{\lambda,z}_s\vert\mathcal{F}^{\xi^{\lambda,z}}_s]=0.
		\end{equation}
		As the integrability of $b^{\lambda,z}$ follows from the integrability of $\xi^{\lambda,z}$, this proves that $b^{\lambda,z}$ is a martingale with respect to $\mathbb{F}^{\xi^{\lambda,z}}$. The last step of proving Theorem \ref{thmsemimartingale} is to show that $b^{\lambda,z}$ is of finite quadratic variation such that $\left\langle b^{\lambda,z}\right\rangle_t=t\wedge\sigma_z^{\lambda}$. It follows from the first step of the proof and \eqref{equationdecompositionsemisigma} that $\left\langle b^{\lambda,z}\right\rangle_t=\left\langle\xi^{\lambda,z}\right\rangle_t$. Hence, it suffices to show that $\left\langle\xi^{\lambda,z}\right\rangle_t=t\wedge\sigma_z^{\lambda}$.
		This is readily proven by the simple equality:
		\begin{equation*}
			\xi_t^{\lambda,z}=B_{t\wedge \sigma_z^{\lambda}}+\lambda\,(t\wedge \sigma_z^{\lambda}),\quad t\geq 0.
		\end{equation*}
		This completes the proof.
	\end{proof}
	\begin{remark}\label{remarksemimartingale}
		Setting $Y^{\lambda}=\xi^{\lambda,z}-z$, for $t<\sigma_z^{\lambda}$, we have
		\begin{equation}
			\mathrm{d}Y^{\lambda}_t=\mathrm{d}b^{\lambda,z}_t-\lambda\,\text{sgn}(Y^{\lambda}_t)\,\mathrm{d}t,
		\end{equation}
		this process is often referred to as bang-bang Brownian motion or Brownian motion with alternating drift, it plays a key role in the extension of P. L\'evy's distributional properties to the case of a Brownian motion with drift. In fact, the difference between the current value and its running maximum has been shown to be equal in law to the absolute value of bang-bang Brownian motion see \cite{GShiryaev}.
	\end{remark}		
	\section{Local Time, Compensator, and Nature of the Last passage Time}
	\label{Sectioncompensator}
	In this section, we analyze the nature of the stopping time $\sigma_z^{\lambda}$ by examining its compensator. Specifically, we aim to determine whether $\sigma_z^{\lambda}$ is predictable, accessible, or totally inaccessible based on the regularity of its compensator. Since the explicit expression of the compensator depends on the local time of $\xi^{\lambda,z}$, we begin by studying the properties of the local time of $\xi^{\lambda,z}$ before deriving and analyzing the compensator of $\sigma_z^{\lambda}$.
	\begin{remark}
		Since $\xi^{\lambda,z}$ is a semi-martingale, its local time $L^{\lambda,z}(t,x)$ at level $x$ up to time $t$ is well defined (see Chapter 5 in \cite{RY}). According to \cite[Theorem 1.7, Ch. IV]{RY}, there exists a modification of the process $(L^{\lambda,z}(t,x),\, t \geq 0,\, x \in \mathbb{R})$ such that the map $(t,x) \in \mathbb{R}_{+} \times \mathbb{R} \mapsto L^{\lambda,z}(t,x)$ is continuous in $t$ and c\`ad-l\`ag in $x$. Moreover, the jump size of $L^{\lambda,z}$ in the spatial variable $x$ is given by
		\begin{equation*}
			L^{\lambda,z}(t,x) - L^{\lambda,z}(t,x-) = 2 \lambda \int_{0}^{t \wedge \sigma_z^{\lambda}} \mathbb{I}_{\{ \xi^{\lambda,z}_s = x \}} \, \text{sgn}(z - \xi^{\lambda,z}_s) \, \mathrm{d}s.
		\end{equation*}
		Applying the occupation times formula to the right-hand side, we obtain
		\begin{equation}
			L^{\lambda,z}(t,x) - L^{\lambda,z}(t,x-) = 2 \lambda \int_{-\infty}^{+\infty} \mathbb{I}_{\{ y = x \}} \, \text{sgn}(z - y) \, L^{\lambda,z}(t,y) \, \mathrm{d}y = 0.
		\end{equation}
		Hence, there exists a version of $L^{\lambda,z}(t,x)$ such that the map $(t,x) \in \mathbb{R}_{+} \times \mathbb{R} \mapsto L^{\lambda,z}(t,x)$ is continuous, $\mathbb{P}$-a.s.
	\end{remark}
	We now analyze the distribution of the local time 
	$L^{\lambda,z}(\sigma_z^{\lambda},z)$ of the process 
	$\xi^{\lambda,z}$ at level $z$ up to the stopping time 
	$\sigma_z^{\lambda}$. The following proposition establishes its explicit law via a direct computation. The distribution of the local time at the last-passage level can also be derived from general results relating dual projections and local times; see, for instance, \cite{AJ}.
	\begin{proposition}\label{propexponentialdistribution}
		The amount of time the process $\xi^{\lambda,z}$ spends at the level $z$ up to $\sigma_z^{\lambda}$, $L^{\lambda,z}(\sigma_z^{\lambda},z)$, follows an exponential distribution with parameter $\lambda$.
	\end{proposition}
	\begin{proof}
		Recall that for $y\geq 0$, $r>0$, and $(a,x)\in \mathbb{R}^2$, we have
		\begin{equation}
			\mathbb{P}(L^B(r,x)>y\vert B_r=a)=\exp\left(-\frac{(\vert x \vert+\vert a-x \vert +y )^2-a^2}{2r}\right),
		\end{equation}
		see \cite{Bo} and \cite{P}. Hence, it follows from \eqref{eqsigmazlamdalaw}, \eqref{eqpredictableformulla}, and \eqref{eqprimitive} that
		\begin{align*}
			\mathbb{P}(L^{\lambda,z}(\sigma_z^{\lambda},z)>y)&=\displaystyle\int_0^{+\infty}\mathbb{P}(L^{\lambda,z}(\sigma_z^{\lambda},z)>y\vert \sigma_z^{\lambda}=r)\mathbb{P}_{\sigma_z^{\lambda}}(\mathrm{d}r)\\
			&=\displaystyle\int_0^{+\infty}\mathbb{P}(L^{B}(r,z)>y\vert B_r=z)\mathbb{P}_{\sigma_z^{\lambda}}(\mathrm{d}r)\\
			&=\lambda\displaystyle\int_0^{+\infty}\exp\left(-\frac{y^2+2yz}{2r}\right)\,p(r,z,\lambda\,r)\,\mathrm{d}r\\
			&=\lambda\,\exp(\lambda z)\displaystyle\int_0^{+\infty}\frac{1}{\sqrt{2\pi r}}\exp\left(-\frac{(y+z)^2}{2r}-\dfrac{\lambda^2}{2}r\right)\,\mathrm{d}r\\
			&=\lambda\,\sqrt{\frac{2}{\pi}}\exp(\lambda z)\displaystyle\int_0^{+\infty}\exp\left(-\frac{(y+z)^2}{2r^2}-\dfrac{\lambda^2}{2}r^2\right)\,\mathrm{d}r\\
			&=\exp(-\lambda\,y).
		\end{align*}
		This shows that $L^{\lambda,z}(\sigma_z^{\lambda},z)$ is exponentially distributed with parameter $\lambda$.
	\end{proof}
	\begin{corollary}
		The law of the supremum of $\xi^{\lambda,z}$ up to time $\sigma_z^{\lambda}$ is the same as the law of the supremum of $z+B^{-\lambda}$, i.e.,
		\begin{equation}
			\sup\limits_{t\leq \sigma_z^{\lambda}} \xi_t^{\lambda,z} \overset{\text{law}}{=} \sup\limits_{t\geq 0} (z+B_t^{-\lambda}).
		\end{equation}
	\end{corollary}
	\begin{proof}
		First, since $\xi_{\sigma_z^{\lambda}}^{\lambda,z}=z$, $\sup\limits_{t\leq \sigma_z^{\lambda}} \xi_t^{\lambda,z}\geq z$. Hence, if $x< z$, then $\mathbb{P}(\sup\limits_{t\leq \sigma_z^{\lambda}} \xi_t^{\lambda,z}\leq x)=0$. If $x\geq z$, we have,
		\begin{align*}
			\mathbb{P}(\sup\limits_{t\leq \sigma_z^{\lambda}} \xi_t^{\lambda,z}\leq x)&=\displaystyle\int_0^{+\infty} \mathbb{P}(\sup\limits_{t\leq \sigma_z^{\lambda}} \xi_t^{\lambda,z}\leq x\vert \sigma_z^{\lambda}=r )\mathbb{P}_{\sigma_z^{\lambda}}(\mathrm{d}r)\\
			&=\displaystyle\int_0^{+\infty} \mathbb{P}(\sup\limits_{t\leq r} B_t^{\lambda}\leq x\vert B_r^{\lambda}=z )\mathbb{P}_{\sigma_z^{\lambda}}(\mathrm{d}r)\\
			&=1-\lambda\displaystyle\int_0^{+\infty} \exp\left(-2\dfrac{x(x-z)}{r}\right)p(r,z,\lambda r)\mathrm{d}r\\
			&=1-\lambda\sqrt{\frac{2}{\pi}}\exp(\lambda z)\displaystyle\int_0^{+\infty} \exp\left(-\dfrac{(2x-z)^2}{2r^2}-\dfrac{\lambda^2}{2}r^2\right)\mathrm{d}r\\
			&=1-\exp(-2\lambda(x-z))
		\end{align*}
		where in the third equality we used the well-known distribution of the maximum of the Brownian bridge see for instance \cite{AW}, \cite{BO}, \cite{BP} and \cite{SW}.
		Which completes the proof.
	\end{proof}
	We now determine the $\mathbb{F}^{\xi^{\lambda,z}}$-predictable compensator  of the stopping time $\sigma_z^{\lambda}$. See for instance \cite{M} for the definition and the properties of the compensator. The following theorem provides its explicit expression, which plays a key role in analyzing the nature of $\sigma_z^{\lambda}$.
	\begin{theorem}\label{thmcompensatorsigma}
		The $\mathbb{F}^{\xi^{\lambda,z}}$-predictable compensator associated with $\sigma_z^{\lambda}$ is
		\begin{equation}
			A_t^{\lambda,z}=\lambda\,L^{\lambda,z}(t\wedge \sigma_z^{\lambda},z).
		\end{equation}
	\end{theorem}
	\begin{proof}
		First from Tanaka's formula and Theorem \ref{thmsemimartingale} we have
		\begin{equation}
			\vert \xi_t^{\lambda,z}-z\vert=z+\displaystyle\int_{0}^{t\wedge \sigma_z^{\lambda}}\text{sgn}(\xi^{\lambda,z}_s-z)\mathrm{d}b^{\lambda,z}_s-\lambda (t\wedge \sigma_z^{\lambda})+L^{\lambda,z}(t\wedge \sigma_z^{\lambda},z).
		\end{equation}
		To show that $A^{\lambda,z}=(\lambda\,L^{\lambda,z}(t\wedge \sigma_z^{\lambda},z), t\geq 0)$ is the compensator of $\sigma_z^{\lambda}$ we need to show that 
		\begin{equation*}
			\mathbb{I}_{\{\sigma_z^{\lambda}\leq t\}}-\lambda\,L^{\lambda,z}(t\wedge \sigma_z^{\lambda},z),\, t\geq 0,
		\end{equation*}
		is an $\mathbb{F}^{\xi^{\lambda,z}}$-martingale. Hence, it is equivalent to show that  
		\begin{equation*}
			\lambda \vert \xi^{\lambda,z}_{\textcolor{blue}{t}}-z\vert + \lambda^2 (t\wedge \sigma_z^{\lambda})+\mathbb{I}_{\{t<\sigma_z^{\lambda}\}},\, t\geq 0,
		\end{equation*}
		is an $\mathbb{F}^{\xi^{\lambda,z}}$-martingale. From \eqref{eqsemigroup}, we have
		\begin{multline*}
			\mathbb{E}\left[\vert \xi_t^{\lambda,z}-z \vert\Big\vert \mathcal{F}^{\xi^{\lambda,z}}_s\right]=\exp(\lambda\,\alpha(z,\xi^{\lambda,z}_{s}))\displaystyle\int_{\mathbb{R}} \vert y-z\vert\exp(-\lambda\,\alpha(z,y) )p(t-s,y-\xi^{\lambda,z}_{s},\lambda(t-s))\hbox{d}y\,\mathbb{I}_{\{\xi^{\lambda,z}_s\neq z\}}\\
			=\exp(\lambda\,\alpha(z,\xi^{\lambda,z}_{s}))\biggl[\displaystyle\int_{-\infty}^z (z-y)p(t-s,y-\xi^{\lambda,z}_{s},\lambda(t-s))\hbox{d}y\\
			+\displaystyle\int_{z}^{+\infty} (y-z)\exp(2\lambda\,(z-y) )p(t-s,y-\xi^{\lambda,z}_{s},\lambda(t-s))\hbox{d}y\biggr]\mathbb{I}_{\{\xi^{\lambda,z}_s\neq z\}}.
		\end{multline*}
		A simple calculation reveals that
		\begin{multline}
			\lambda\mathbb{E}\left[\vert \xi_t^{\lambda,z}-z \vert\Big\vert \mathcal{F}^{\xi^{\lambda,z}}_s\right]=\biggl\{2\lambda(t-s)\exp(\lambda\alpha(z,\xi^{\lambda,z}_s))p(t-s,z-\xi^{\lambda,z}_s,\lambda\,(t-s))\\+(\lambda(z-\xi^{\lambda,z}_{s})-\lambda^2(t-s))\exp(\lambda\alpha(z,\xi^{\lambda,z}_s))\Phi\bigg(-\lambda\sqrt{t-s}+\dfrac{ z-\xi^{\lambda,z}_{s}}{\sqrt{t-s}}\bigg)\\
			-(\lambda(z-\xi^{\lambda,z}_{s})+\lambda^2(t-s))\exp(\lambda\gamma(z,\xi^{\lambda,z}_s))\Phi\bigg(-\lambda\sqrt{t-s}- \dfrac{ z-\xi^{\lambda,z}_{s}}{\sqrt{t-s}}\bigg)\biggr\}\mathbb{I}_{\{\xi^{\lambda,z}_s\neq z\}}.\label{eqlambdaEvert}
		\end{multline}
		On the other hand, from \eqref{eqt<sigmagivenxi_s} and \eqref{eqsigmas<sigma<tgivinxi_s}, we have for all $s\leq t$,
		\begin{multline}
			\lambda^2\,\mathbb{E}[t\wedge \sigma_z^{\lambda}\vert \mathcal{F}_{s}^{\xi^{\lambda,z}}]=\biggl\{\Big(\lambda^2\,(t-s)-1-\lambda(z-\xi^{\lambda,z}_{s})\Big)\Phi\bigg(-\lambda\sqrt{t-s}+\dfrac{ z-\xi^{\lambda,z}_{s} }{\sqrt{t-s}}\bigg)\\
			+\Big(\lambda^2\,(t-s)-1+\lambda(z-\xi^{\lambda,z}_{s})\Big)\exp(2\lambda\, (z-\xi^{\lambda,z}_{s}))\Phi\bigg(-\lambda\sqrt{t-s}- \dfrac{ z-\xi^{\lambda,z}_{s}}{\sqrt{t-s}}\bigg)\\
			-2\lambda(t-s)p(t-s,z-\xi^{\lambda,z}_s,\lambda\,(t-s))+\lambda^2\,s+\lambda(z-\xi^{\lambda,z}_s)+1\biggr\}\mathbb{I}_{\{\xi^{\lambda,z}_s< z\}}\\
			+\biggl\{\Big(\lambda^2\,(t-s)-1-\lambda(z-\xi^{\lambda,z}_{s})\Big)\exp(-2\lambda\, (z-\xi^{\lambda,z}_{s}))\Phi\bigg(-\lambda\sqrt{t-s}+ \dfrac{ z-\xi^{\lambda,z}_{s}}{\sqrt{t-s}}\bigg)\\
			+\Big(\lambda^2\,(t-s)-1+\lambda(z-\xi^{\lambda,z}_{s})\Big)\Phi\bigg(-\lambda\sqrt{t-s}- \dfrac{ z-\xi^{\lambda,z}_{s} }{\sqrt{t-s}}\bigg)+1+\lambda^2\,s+\lambda(\xi^{\lambda,z}_s-z)\\
			-2\lambda(t-s)\exp(-2\lambda\, (z-\xi^{\lambda,z}_{s}))p(t-s,z-\xi^{\lambda,z}_s,\lambda\,(t-s))\biggr\}\mathbb{I}_{\{\xi^{\lambda,z}_s> z\}}
			+\lambda^2\sigma_z^{\lambda}\mathbb{I}_{\{\xi^{\lambda,z}_s= z\}},\label{eqEsplit3}
		\end{multline}
		and
		\begin{multline}
			\mathbb{P}(t<\sigma_z^{\lambda}\vert \mathcal{F}_{s}^{\xi^{\lambda,z}})=\biggl\{\Phi\bigg(-\lambda\sqrt{t-s}+\dfrac{ z-\xi^{\lambda,z}_{s}}{\sqrt{t-s}}\bigg)+\exp(2\lambda\,( z-\xi^{\lambda,z}_{s}))\\\times\Phi\bigg(-\lambda\sqrt{t-s}-\dfrac{ z-\xi^{\lambda,z}_{s}}{\sqrt{t-s}}\bigg)\biggr\}\mathbb{I}_{\{\xi^{\lambda,z}_s< z\}}
			+\biggl\{\Phi\bigg(-\lambda\sqrt{t-s}-\dfrac{ z-\xi^{\lambda,z}_{s}}{\sqrt{t-s}}\bigg)\\+\exp(-2\lambda\, (z-\xi^{\lambda,z}_{s}))\Phi\bigg(-\lambda\sqrt{t-s}+\dfrac{ z-\xi^{\lambda,z}_{s}}{\sqrt{t-s}}\bigg)\biggr\}\mathbb{I}_{\{\xi^{\lambda,z}_s> z\}}.\label{eqt<sigmagivenxi_s1}
		\end{multline}
		Thus, it follows from \eqref{eqlambdaEvert}, \eqref{eqEsplit3} and \eqref{eqt<sigmagivenxi_s1} that for all $s\leq t$,
		\begin{align*}
			&\mathbb{E}\left[\lambda \vert \xi^{\lambda,z}_t-z\vert + \lambda^2 (t\wedge \sigma_z^{\lambda})+\mathbb{I}_{\{t<\sigma_z^{\lambda}\}}\Big\vert\mathcal{F}_{s}^{\xi^{\lambda,z}}\right]\\
			&=\lambda\mathbb{E}\left[\vert \xi^{\lambda,z}_t-z\vert\Big\vert\mathcal{F}_{s}^{\xi^{\lambda,z}}\right]+\lambda^2\mathbb{E}[t\wedge \sigma_z^{\lambda}\vert\mathcal{F}_{s}^{\xi^{\lambda,z}}]+\mathbb{P}(t<\sigma_z^{\lambda}\vert \mathcal{F}_{s}^{\xi^{\lambda,z}})\\
			&=\Bigl\{\lambda(z-\xi^{\lambda,z}_s)+\lambda^2\,s+1\Bigr\}\mathbb{I}_{\{\xi^{\lambda,z}_s< z\}}+\Bigl\{\lambda(\xi^{\lambda,z}_s-z)+\lambda^2\,s+1\Bigr\}\mathbb{I}_{\{\xi^{\lambda,z}_s> z\}}+\lambda^2\sigma_z^{\lambda}\mathbb{I}_{\{\xi^{\lambda,z}_s= z\}}\\
			&=\Bigl\{\lambda\vert\xi^{\lambda,z}_s-z\vert+\lambda^2\,s+1\Bigr\}\mathbb{I}_{\{\xi^{\lambda,z}_s\neq z\}}+\Bigl\{\lambda\vert \xi^{\lambda,z}_s-z\vert+\lambda^2\sigma_z^{\lambda}\Bigr\}\mathbb{I}_{\{\xi^{\lambda,z}_s= z\}}\\
			&=\Bigl\{\lambda\vert\xi^{\lambda,z}_s-z\vert+\lambda^2\,s+1\Bigr\}\mathbb{I}_{\{s<\sigma_z^{\lambda}\}}+\Bigl\{\lambda\vert\xi^{\lambda,z}_s-z\vert+\lambda^2\sigma_z^{\lambda}\Bigr\}\mathbb{I}_{\{\sigma_z^{\lambda}\leq s\}}\\
			&=\Bigl\{\lambda\vert\xi^{\lambda,z}_s-z\vert+\lambda^2(s\wedge \sigma_z^{\lambda})+\mathbb{I}_{\{s<\sigma_z^{\lambda}\}}\Bigr\}\mathbb{I}_{\{s<\sigma_z^{\lambda}\}}+\Bigl\{\lambda\vert \xi^{\lambda,z}_s-z\vert+\lambda^2(s\wedge \sigma_z^{\lambda})+\mathbb{I}_{\{s<\sigma_z^{\lambda}\}}\Bigr\}\mathbb{I}_{\{\sigma_z^{\lambda}\leq s\}}\\
			&=\lambda\vert \xi^{\lambda,z}_s-z\vert+\lambda^2(s\wedge \sigma_z^{\lambda})+\mathbb{I}_{\{s<\sigma_z^{\lambda}\}}.
		\end{align*}
		As the integrability of $L^{\lambda,z}(t,z)$ follows from the integrability of $\xi_t^{\lambda,z}$, this proves that
		\begin{equation*}
			\mathbb{I}_{\{\sigma_z^{\lambda}\leq t\}}-\lambda\,L^{\lambda,z}(t\wedge \sigma_z^{\lambda},z),\, t\geq 0,
		\end{equation*}
		is a martingale with respect to $\mathbb{F}^{\xi^{\lambda,z}}$. Hence, the process $A^{\lambda,z}$ is the compensator of $\sigma_z^{\lambda}$.
	\end{proof}
	\begin{corollary}\label{cortotallyinaccessible}
		The random time $\sigma_z^{\lambda}$ is a totally inaccessible stopping time with respect to the filtration $\mathbb{F}^{\xi^{\lambda,z},c}$.
	\end{corollary}
	\begin{proof}
		Given that the compensator $A^{\lambda,z}$ of $\sigma_z^{\lambda}$ is continuous, it can be concluded that $\sigma_z^{\lambda}$ is a totally inaccessible stopping time with respect to the filtration $\mathbb{F}^{\xi^{\lambda,z},c}$. This conclusion follows from the well-established result that the continuity of the compensator is equivalent to $\sigma_z^{\lambda}$ being totally inaccessible, as detailed in \cite[Corollary 25.18]{K}.
	\end{proof}
	\begin{remark}
		We can deduce the result of Proposition \ref{propexponentialdistribution} from Corollary \ref{cortotallyinaccessible} by using the fact that if a stopping time is totally inaccessible, its compensator at infinity, $A^{\lambda,z}_{\infty}$, is exponentially distributed with rate parameter $1$. Consequently, since $A_t^{\lambda,z}=\lambda\,L^{\lambda,z}(t\wedge \sigma_z^{\lambda},z)$, it follows that the random variable $\lambda\,L^{\lambda,z}(\sigma_z^{\lambda},z)$ follows an exponential distribution  with rate parameter $1$. Thus, $L^{\lambda,z}(\sigma_z^{\lambda},z)$ follows an exponential distribution  with rate parameter $\lambda$. This implies that the expected amount of time the process $\xi^{\lambda,z}$ spends at the level $z$ up to time $\sigma_z^{\lambda}$ is $\frac{1}{\lambda}$.
	\end{remark}
	\section{Accessible, Predictable Times and Quasi-left-continuity}
	\label{Sectionaccessiblepredictabletimes}
	In this section, we study the stopping times within the filtration generated by the process $\xi^{\lambda,z}$. Firstly, we show that every accessible stopping time is predictable. This is equivalent to the fact that the filtration $\mathbb{F}^{\xi^{\lambda,z},c}$ is quasi-left-continuous, that is, $\mathcal{F}^{\xi^{\lambda,z},c}_{T-}=\mathcal{F}^{\xi^{\lambda,z},c}_{T}$, for every predictable stopping time $T$. From a modeling point of view, a quasi-left-continuous filtration ensures that the only jumps or surprises in the model come from totally inaccessible events. Secondly, we show that the last passage time $\sigma_z^{\lambda}$ is the only stopping time totally inaccessible within the filtration $\mathbb{F}^{\xi^{\lambda,z},c}$, in the sense that every stopping time $T$ such that $\mathbb{P}(T=\sigma_z^{\lambda})=0$, the stopping time $T$ is predictable. 
	\begin{theorem}\label{thmquasi-left-continuous}
		The completed filtration generated by $\xi^{\lambda,z}$, $\mathbb{F}^{\xi^{\lambda,z},c}$, is quasi-left-continuous.
	\end{theorem}
	It is known that the filtration generated by a Feller process is quasi-left-continuous. But, $\xi^{\lambda,z}$ is not a Feller process, see Remark \ref{remarkfeller}. Even though we can see that what ruin the Feller property is the fact that the information about $\xi^{\lambda,z}_t$ at time $t$ is not enough to determine the location of $\tau(\omega)$ with respect to $t$ whether $\tau(\omega) \leq t$ or $t < \tau(\omega)$ for a given outcome $\omega$ in the sample space $\Omega$. To overcome this problem we consider the two-dimensional process $\zeta^{\lambda,z}$ defined by
	\begin{equation}
		\zeta^{\lambda,z}_t=(\mathbb{I}_{\{t<\sigma_z^{\lambda}\}},\xi_t^{\lambda,z}),\,t\geq 0
	\end{equation}
	The process $\zeta^{\lambda,z}$ belongs to the set of values represented by 
	\begin{equation*}
		\mathcal{S}=\{0,1\}\times\mathbb{R}.
	\end{equation*}
	From \eqref{eqmodification}, we have $\mathbb{I}_{\{t<\sigma_z^{\lambda}\}}=\mathbb{I}_{\{\xi_t^{\lambda,z}\neq z\}}$, $\mathbb{P}$-a.s., hence, 
	\begin{equation}
		\mathcal{F}_t^{\xi^{\lambda,z},c}=\mathcal{F}_t^{\zeta^{\lambda,z},c}, t\geq 0.\label{eqcoincidefiltrations}
	\end{equation}
	Thus, the proof of Theorem \ref{thmquasi-left-continuous} is an immediate consequence of the following theorem:
	\begin{theorem}\label{thmFeller}
		The process $\zeta^{\lambda,z}$ is a Feller process, that is, its semigroup $Q^{\lambda,z}$ satisfies the following statements:
		\begin{enumerate}
			\item[(i)] $Q^{\lambda,z}_t h(y) \rightarrow h(y)$ as $t \rightarrow 0, h \in C_0(\{0,1\}\times \mathbb{R}), y=(y_1, y_2) \in \{0,1\}\times \mathbb{R}$
			\item[(ii)] $ Q^{\lambda,z}_t C_0(\{0,1\}\times \mathbb{R}) \subset C_0(\{0,1\}\times \mathbb{R}), t \geq 0$.
		\end{enumerate}
	\end{theorem}
	\begin{proof}
		Firstly, we remark that $\mathbb{F}^{\xi^{\lambda,z},c}=\mathbb{F}^{\zeta^{\lambda,z},c}$. We claim that the process $\zeta^{\lambda,z}$ is Markovian with semigroup $Q^{\lambda,z}$ given by
		\begin{multline}
			Q^{\lambda,z}_th(y_1,y_2)=(1-y_1)h(y_1,y_2)
			+y_1\,h(0,z)\bigg[\Phi\Big(\lambda\sqrt{t}- \dfrac{\vert z-y_2 \vert}{\sqrt{t}}\Big)-\exp(2\lambda\,\vert z-y_2 \vert)\Phi\Big(-\lambda\sqrt{t}- \dfrac{\vert z-y_2 \vert}{\sqrt{t}}\Big)\bigg]\\+y_1\,\exp(\lambda\,\alpha(z,y_2))\displaystyle\int_{\mathbb{R}} h(1,\eta)\exp(-\lambda\,\alpha(z,\eta) )p(t,\eta-y_2,\lambda\,t)\hbox{d}\eta \label{eqQsemigroup}
		\end{multline}
		for any $t\geq 0$ and for every bounded measurable function $h$ defined on $\{0,1\}\times \mathbb{R}$. Indeed, Theorem \ref{thmMarkov} yields, $\mathbb{P}$-a.s.,
		\begin{align*}
			\mathbb{E}[h(\zeta_{t+s}^{\lambda,z})\vert\mathcal{F}_s^{\xi^{\lambda,z},c}]
			&=h(0,z)\mathbb{I}_{\{\sigma_z^{\lambda}\leq s\}}+\mathbb{E}[h(0,z)\mathbb{I}_{\{s<\sigma_z^{\lambda}\leq s+t\}}\vert\mathcal{F}_s^{\xi^{\lambda,z},c}]+\mathbb{E}[h(1,\xi_{s+t}^{\lambda,z})\mathbb{I}_{\{s+t<\sigma_z^{\lambda}\}}\vert\mathcal{F}_s^{\xi^{\lambda,z},c}]\\
			&=h(0,z)\mathbb{I}_{\{\sigma_z^{\lambda}\leq s\}}+\mathbb{E}[h(0,z)\mathbb{I}_{\{\xi_{s+t}^{\lambda,z}=z\}}\vert\xi_s^{\sigma_z^{\lambda}}]\mathbb{I}_{\{s<\sigma_z^{\lambda}\}}+\mathbb{E}[h(1,\xi_{s+t}^{\lambda,z})\mathbb{I}_{\{\xi_{s+t}^{\lambda,z}\neq z\}}\vert\xi_{s}^{\lambda,z}]\\
			&=h(0,z)(1-\mathbb{I}_{\{s<\sigma_z^{\lambda}\}})+h(0,z)\bigg[ \Phi\bigg(\lambda\sqrt{t}- \dfrac{\vert z-\xi^{\lambda,z}_{s} \vert}{\sqrt{t}}\bigg)-\exp(2\lambda\,\vert z-\xi^{\lambda,z}_{s} \vert)\nonumber\\
			&\times\Phi\bigg(-\lambda\sqrt{t}- \dfrac{\vert z-\xi^{\lambda,z}_{s} \vert}{\sqrt{t}}\bigg)\bigg]\mathbb{I}_{\{ s<\sigma_z^{\lambda} \}}+\exp(\lambda\,\alpha(z,\xi^{\lambda,z}_{s}))\nonumber\\
			&\times\displaystyle\int_{\mathbb{R}} h(1,\eta)\exp(-\lambda\,\alpha(z,\eta) )p(t,y-\xi^{\lambda,z}_{s},\lambda\,t)\mathrm{d}y\,\mathbb{I}_{\{s< \sigma_z^{\lambda}\}}\\
			&=Q^{\lambda,z}_th(\zeta_s^{\lambda,z}).
		\end{align*}
		Proof of statement (i): It is clear that for $y_2\neq z$
		\begin{equation*}
			\lim\limits_{t\rightarrow 0} \Phi\Big(\lambda\sqrt{t}- \dfrac{\vert z-y_2 \vert}{\sqrt{t}}\Big)=\lim\limits_{t\rightarrow 0} \Phi\Big(-\lambda\sqrt{t}- \dfrac{\vert z-y_2 \vert}{\sqrt{t}}\Big)=0.
		\end{equation*}
		and for $y_2=z$
		\begin{equation*}
			\lim\limits_{t\rightarrow 0} \Phi\Big(\lambda\sqrt{t}- \dfrac{\vert z-y_2 \vert}{\sqrt{t}}\Big)=\lim\limits_{t\rightarrow 0} \Phi\Big(-\lambda\sqrt{t}- \dfrac{\vert z-y_2 \vert}{\sqrt{t}}\Big)=1/2.
		\end{equation*}
		This yields that for any $y_2\in \mathbb{R}$,
		\begin{equation*}
			\lim\limits_{t\rightarrow 0} \bigg[\Phi\Big(\lambda\sqrt{t}- \dfrac{\vert z-y_2 \vert}{\sqrt{t}}\Big)-\exp(2\lambda\,\vert z-y_2 \vert)\Phi\Big(-\lambda\sqrt{t}- \dfrac{\vert z-y_2 \vert}{\sqrt{t}}\Big)\bigg]=0.
		\end{equation*}
		On the other hand, the functions $p(t, \cdot, y_2+\lambda\,t)$ are probability density functions with respect to the Lebesgue measure on $\mathbb{R}$. The probability measures $\mathbb{Q}_t$ on $\mathbb{R}$ associated with the density $p(t, \cdot, y_2+\lambda\,t)$ converge weakly as $t \downarrow 0$ to the Dirac measure $\delta_{y_2}$ at $y_2$. Moreover, the function $x\longrightarrow h(1,x)\exp(-\lambda\,\alpha(z,x))$ is bounded and continuous. Hence, we can pass to the limit and obtain the following
		\begin{equation*}
			\lim\limits_{t\rightarrow 0}\displaystyle\int_{\mathbb{R}} h(1,\eta)\exp(-\lambda\,\alpha(z,\eta) )p(t,\eta-y_2,\lambda\,t)\hbox{d}\eta=h(1,y_2)\,\exp(-\lambda\,\alpha(z,y_2)).
		\end{equation*}
		Thus, 
		\begin{equation*}
			\lim\limits_{t\rightarrow 0}\,Q^{\lambda,z}_th(y)=(1-y_1)h(y_1,y_2)+y_1h(1,y_2)=h(y).
		\end{equation*}
		Which completes the proof of the statement (i).\\
		Proof of statement (ii): Let $h\in C_0(\{0,1\}\times \mathbb{R})$. The function $y\longrightarrow Q^{\lambda,z}_th(y)$ is continuous on $\{0,1\}\times \mathbb{R}$ since $y_2\longrightarrow Q^{\lambda,z}_th(0,y_2)$ and $y_2\longrightarrow Q^{\lambda,z}_th(1,y_2)$ are continuous on $\mathbb{R}$. Let us now show that for all $t\geq 0$
		\begin{equation}
			\lim\limits_{\vert y\vert_2\rightarrow +\infty}Q^{\lambda,z}_th(y)=0.
		\end{equation}
		Note that $\vert y\vert_{2}^2=\vert y_1\vert^2+\vert y_2\vert^2$, since $y_1\in \{0,1\}$, $\vert y\vert_{2}\rightarrow +\infty$ if and only if $\vert y_2\vert\rightarrow +\infty$. Thus, it is sufficient to show that for all $t\geq 0$,
		\begin{equation*}
			\lim\limits_{\vert y_2\vert\rightarrow +\infty}Q^{\lambda,z}_th(y_1,y_2)=0.
		\end{equation*}
		For $t=0$, the above statement is true since $\vert Q^{\lambda,z}_0h(y_1,y_2)\vert=\vert h(y_1,y_2)\vert\rightarrow 0$ as $\vert y_2\vert\rightarrow +\infty$. Assume that $t>0$. Let $\varepsilon > 0$.
		It follows from \eqref{eqQsemigroup} that
		\begin{multline}
			\vert Q^{\lambda,z}_th(y_1,y_2) \vert \leq\vert h(y_1,y_2) \vert 
			+ \vert\vert h \vert\vert_{\infty}\bigg[\Phi\Big(\lambda\sqrt{t}- \dfrac{\vert z-y_2 \vert}{\sqrt{t}}\Big)+\exp(2\lambda\,\vert z-y_2 \vert)\Phi\Big(-\lambda\sqrt{t}- \dfrac{\vert z-y_2 \vert}{\sqrt{t}}\Big)\bigg]\\
			+\exp(\lambda\,\alpha(z,y_2))\displaystyle\int_{\mathbb{R}} \vert h(1,\eta)\vert \exp(-\lambda\,\alpha(z,\eta) )p(t,\eta-y_2,\lambda\,t)\hbox{d}\eta.\label{eqFeller00}
		\end{multline}
		Setting $\bar{x}=\lambda\sqrt{t}+\dfrac{\vert z-y_2 \vert}{\sqrt{t}}$, we have 
		\begin{equation*}
			\Phi(-\bar{x})=\frac{1}{\sqrt{2 \pi}} \int_{\bar{x}}^{+\infty} \exp\Big(-\dfrac{x^2}{2}\Big) \mathrm{d} x \leq\frac{1}{\sqrt{2 \pi}} \int_{\bar{x}}^{+\infty} \frac{x}{\bar{x}} \exp\Big(-\dfrac{x^2}{2}\Big) \mathrm{d} x=\frac{1}{\sqrt{2 \pi}} \frac{1}{\bar{x}} \exp\Big(-\dfrac{\bar{x}^2}{2}\Big).
		\end{equation*}
		Hence, 
		\begin{multline}
			\exp(2\lambda\,\vert z-y_2 \vert)\Phi\Big(-\lambda\sqrt{t}-\dfrac{\vert z-y_2 \vert}{\sqrt{t}}\Big)\leq\\
			\frac{1}{\sqrt{2 \pi}} \dfrac{\sqrt{t}}{\lambda\,t+\vert z-y_2 \vert}\exp\bigg(-\dfrac{1}{2}\bigg(\lambda\sqrt{t}-\dfrac{\vert z-y_2 \vert}{\sqrt{t}}\bigg)^2\bigg).\label{eqFeller1}
		\end{multline}
		On the other hand, change variables yields
		\begin{multline}
			I(y_2):=\exp(\lambda\,\alpha(z,y_2))\displaystyle\int_{\mathbb{R}} \vert h(1,\eta)\vert \exp(-\lambda\,\alpha(z,\eta) )p(t,\eta-y_2,\lambda\,t)\hbox{d}\eta=\\
			\displaystyle\int_{\mathbb{R}} \vert h(1,y_2+\eta)\vert \exp(\lambda(\alpha(z,\eta)-\alpha(z,y_2+\eta)) )p(t,\eta,\lambda\,t)\hbox{d}\eta
		\end{multline}
		We have for all $y_2$, $\eta$
		\begin{equation*}
			\vert \alpha(z,\eta)-\alpha(z,y_2+\eta)\vert \leq 2\vert \eta \vert,
		\end{equation*}
		so
		\begin{equation*}
			\vert h(1,y_2+\eta)\vert \exp(\lambda(\alpha(z,\eta)-\alpha(z,y_2+\eta)) )p(t,\eta,\lambda\,t)\leq \vert\vert h \vert\vert_{\infty} \exp(2\lambda\,\vert \eta \vert) p(t,\eta,\lambda\,t).
		\end{equation*}
		The right-hand side is an integrable function of $\eta$ (the Gaussian factor dominates the exponential $\exp(2\lambda\,\vert \eta \vert)$), so by dominated convergence the integrand tends to $0$ for each fixed $\eta$ and we may pass the limit inside the integral. Hence
		\begin{equation*}
			\lim\limits_{\vert y\vert_2\rightarrow +\infty}I(y_2)=0.
		\end{equation*}
		Which completes the proof of the statement (ii). 
	\end{proof}
	\begin{remark}
		The statements (i) and (ii) together with the semigroup property imply the strong continuity $Q^{\lambda,z}_t h \rightarrow h$ as $t \rightarrow 0, h \in C_0(\{0,1\}\times \mathbb{R})$, see \cite[Theorem 19.6]{K}.
	\end{remark}
	\begin{theorem}\label{thmonlyTIST}
		The last passage time $\sigma_z^{\lambda}$ is the only stopping time totally inaccessible within the filtration $\mathbb{F}^{\xi^{\lambda,z},c}$. That is, for every stopping time $T$, we have $\mathbb{P}(T=\sigma_z^{\lambda})=0$ if and only if the stopping time $T$ is predictable.  
	\end{theorem}		
	\begin{proof}
		The proof is an immediate consequence of Theorem \ref{thmFeller}, Theorem \ref{thmappendixpredictable}, the equality \eqref{eqcoincidefiltrations} and the fact that the process $\zeta^{\lambda,z}$ jumps only at $\sigma_z^{\lambda}$.
	\end{proof}
	\begin{remark}
		Let $T$ be a stopping time with respect to $\mathbb{F}^{\xi^{\lambda,z},c}$. 
		Then $T=T_{\{T=\sigma_z^{\lambda}\}}\wedge T_{\{T\neq \sigma_z^{\lambda}\}}$, where $T_{\{T=\sigma_z^{\lambda}\}}$ is totally inaccessible and $T_{\{T\neq \sigma_z^{\lambda}\}}$ is predictable, where $T_{\{T=\sigma_z^{\lambda}\}}$ and $T_{\{T\neq \sigma_z^{\lambda}\}}$ are the restrictions of $T$ to $\{T=\sigma_z^{\lambda}\}$ and $\{T\neq \sigma_z^{\lambda}\}$, respectively. 
	\end{remark}
	\section{Generator and Its Associated PDE}
	\label{sectiongenerator}
	In the previous sections, we have shown that the process \(\xi^{\lambda,z}\) is Markovian but not strong Markov. We have provided its canonical decomposition as a semi-martingale in its own (completed) filtration. To capture both the trajectory of \(\xi^{\lambda,z}\) and the occurrence of the stopping time \(\sigma_z^\lambda\), we introduced the process
	\[
	\zeta^{\lambda,z}_t = \left(\mathbb{I}_{\{t<\sigma_z^\lambda\}}, \xi_t^{\lambda,z}\right),
	\]
	and established that it defines a Feller process on the state space \(\mathcal{S} = \{0,1\} \times \mathbb{R}\). 
	In this section, we determine the infinitesimal generator $A^{\lambda,z}$ of the Feller process $\zeta^{\lambda,z}$. This operator encodes the local dynamics of the process and plays a central role in both its probabilistic and analytic study. In particular, it provides a natural framework for constructing martingales and for deriving the associated partial differential equations (see, e.g., \cite{RY}). Moreover, such operators also arise in the analysis of last-passage times for linear diffusions; see, for instance, \cite{CCS}.
	\begin{theorem}\label{thmgenerator}
		The infinitesimal generator $A^{\lambda,z}$ of $\zeta^{\lambda,z}$ is given by
		\begin{align*}
			\begin{array}{ll}
				A^{\lambda,z}h(y_1,y_2)=\begin{cases} 0 & y_1=0, y_2\in \mathbb{R} \\
					-\lambda y_1\,\dfrac{\partial h}{\partial y_2}(y_1,y_2) +  \dfrac{y_1}{2}\,\dfrac{\partial^2 h}{\partial ^2y_2}(y_1,y_2) &  y_1=1, z<y_2 \\
					\dfrac{y_1}{2}\,\dfrac{\partial^2 h}{\partial ^2y_2}(y_1,y_2) &  y_1=1, y_2=z \\
					\lambda y_1\,\dfrac{\partial h}{\partial y_2}(y_1,y_2) +  \dfrac{y_1}{2}\,\dfrac{\partial^2 h}{\partial ^2y_2}(y_1,y_2) &  y_1=1, z>y_2
				\end{cases}
			\end{array}
		\end{align*}
		Moreover, the domain of $A^{\lambda,z}$ is given by
		\begin{equation*}
			\mathcal{D}(A^{\lambda,z})=\left\{h\in C_0(\mathcal{S}); \dfrac{\partial h}{\partial y_2}, \text{ and } \dfrac{\partial^2 h}{\partial ^2y_2} \text{ exist,  continuous, }h(0,z)=h(1,z)=\dfrac{\partial h}{\partial y_2}(1,z)=0 \right\}.
		\end{equation*}
	\end{theorem}
	\begin{proof}
		A function $h \in C_0(\mathcal{S})$ is said to belong to the domain $\mathcal{D}(A^{\lambda,z})$ of the infinitesimal generator of $\zeta^{\lambda,z}$ if the limit
		$$
		A^{\lambda,z} h=\lim _{t \rightarrow 0} \frac{Q_t h-h}{t}
		$$
		exists in $C_0(\mathcal{S})$. We have,
		\begin{align}
			\dfrac{Q_th(y_1,y_2)-h(y_1,y_2)}{t}&=y_1\,h(0,z)\dfrac{\Phi\Big(\lambda\sqrt{t}- \frac{\vert z-y_2 \vert}{\sqrt{t}}\Big)-\exp(2\lambda\,\vert z-y_2 \vert)\Phi\Big(-\lambda\sqrt{t}- \frac{\vert z-y_2 \vert}{\sqrt{t}}\Big)}{t}\nonumber\\&+y_1\,\exp(\lambda\,\alpha(z,y_2))\dint_{\mathbb{R}} \dfrac{g(\eta)-g(y_2)}{t}p(t,\eta-y_2,\lambda\,t)\hbox{d}\eta \label{eq(Q_t-Q_0)/t}
		\end{align}
		where 
		\begin{equation}
			g(x)=\exp(-\lambda\,\alpha(z,x) )h(1,x).
		\end{equation}
		Thus, if $y_1=1$ and $y_2\neq z$, we obtain 
		\begin{align*}
			\lim\limits_{t\rightarrow 0}\dfrac{Q_th(y_1,y_2)-h(y_1,y_2)}{t}&=y_1\,\exp(\lambda\,\alpha(z,y_2))\Big(\lambda g{'}(y_2)+\dfrac{1}{2}g{''}(y_2)\Big),
		\end{align*}
		where,
		\begin{equation}
			\begin{array}{ll}
				g{'}(y_2)=\begin{cases}\left(\frac{\partial h}{\partial y_2}(y_1,y_2)-2 \lambda h(y_1,y_2)\right)\exp(2\lambda\,(z-y_2)) & z<y_2 \\
					\frac{\partial h}{\partial y_2}(y_1,y_2) & z>y_2
				\end{cases}
			\end{array}
		\end{equation}
		and
		\begin{equation}
			\begin{array}{ll}
				g{''}(y_2)=\begin{cases}\left(\frac{\partial^2 h}{\partial ^2y_2}(y_1,y_2)-4 \lambda \frac{\partial h}{\partial y}(y_1,y_2)+4\lambda^2h(y_1,y_2)\right)\exp(2\lambda\,(z-y_2)) & z<y_2 \\
					\frac{\partial^2 h}{\partial ^2y_2}(y_1,y_2) & z>y_2
				\end{cases}.
			\end{array}
		\end{equation}
		If $y_1=1$ and $y_2= z$, the equation \eqref{eq(Q_t-Q_0)/t} takes the form
		\begin{align}
			\dfrac{Q_th(y_1,y_2)-h(y_1,y_2)}{t}&=h(0,z)\dfrac{\Phi(\lambda\sqrt{t})-\Phi(-\lambda\sqrt{t})}{t}+\dint_{\mathbb{R}} \dfrac{g(\eta)-g(z)}{t}p(t,\eta-z,\lambda\,t)\hbox{d}\eta.\label{eqy_1=1,y_2=z}
		\end{align}
		The second term in \ref{eqy_1=1,y_2=z} decomposes as
		\begin{multline*}
			\dint_{\mathbb{R}} \dfrac{g(\eta)-g(z)}{t}p(t,\eta-z,\lambda\,t)\hbox{d}\eta=\dint_{-\infty}^{z} \dfrac{h(1,\eta)-h(1,z)}{t}p(t,\eta-z,\lambda\,t)\hbox{d}\eta\\
			+\dint_{z}^{+\infty} \dfrac{\exp(2\lambda (z-\eta))h(1,\eta)-h(1,z)}{t}p(t,\eta-z,\lambda\,t)\hbox{d}\eta.
		\end{multline*}
		A direct computation shows that the limit of $\frac{Q_t h-h}{t}$ exits as $t$ approaches $0$ if and only if $h(0,z)=h(1,z)=\frac{\partial h}{\partial y_2}(1,z)=0$, in which case the limit is
		$$\dfrac{y_1}{2}\,\dfrac{\partial^2 h}{\partial ^2y_2}(y_1,y_2).$$
		This completes the proof.
	\end{proof}
	\begin{remark}
		The generator $A^{\lambda,z}$ reflects the continuous nature of the process $\xi^{\lambda,z}$, which evolves as a Brownian motion with drift $\lambda$ until the last time it reaches the level $z$. The auxiliary process $\zeta^{\lambda,z}_t = (\mathbb{I}_{\{t<\sigma_z^\lambda\}}, \xi_t^{\lambda,z})$ encodes both the position and whether the process has been stopped, enabling a Feller process description on the state space $\mathcal{S}=\{0,1\} \times \mathbb{R}$. The generator acts non-trivially only on the active state $(1, y_2)$: when $y_2 > z$, the drift term is negative (pulling back toward $z$), and when $y_2 < z$, the drift is positive (pushing up toward $z$), which reflects the conditioning on hitting $z$ again. At the level $z$, the drift vanishes, and only the diffusion term remains. The domain condition $h(0,z) = h(1,z) = \frac{\partial h}{\partial y_2}(1,z) = 0$ ensures smoothness at the boundary and aligns with the fact that the process is absorbed at $z$ after the random time $\sigma_z^\lambda$.
	\end{remark}
	Infinitesimal generators provide a key link between Feller processes and martingales. Knowing the generator of a Feller process allows for the construction of various important martingales. The following result illustrates this connection:
	\begin{theorem}
		For every $h \in \mathcal{D}(A^{\lambda,z})$, the process
		\begin{multline}
			M_t^{\lambda,z}=h(\mathbb{I}_{\{t<\sigma_z^{\lambda}\}},\xi_t^{\lambda,z})-h(1,0)-\displaystyle\int_0^{t} \left(-\lambda\,\dfrac{\partial h}{\partial y_2}(1,\xi_s^{\lambda,z}) +  \dfrac{1}{2}\,\dfrac{\partial^2 h}{\partial ^2y_2}(1,\xi_s^{\lambda,z})\right) \mathbb{I}_{\{z<\xi_s^{\lambda,z}\}} \mathrm{d} s\\
			-\displaystyle\int_0^{t} \left(\lambda\,\dfrac{\partial h}{\partial y_2}(1,\xi_s^{\lambda,z}) +  \dfrac{1}{2}\,\dfrac{\partial^2 h}{\partial ^2y_2}(1,\xi_s^{\lambda,z})\right) \mathbb{I}_{\{\xi_s^{\lambda,z}<z\}} \mathrm{d} s
		\end{multline}
		is a martingale with respect to $\mathbb{F}^{\xi^{\lambda,z}}$.
	\end{theorem}
	\begin{proof}
		Immediate consequence of Theorem \ref{thmgenerator} and the martingale problem.
	\end{proof}
	\begin{remark}
		By Dynkin's formula, for every $\mathbb{F}^{\xi^{\lambda,z},c}$-stopping time $\tau$ with $\mathbb{E}[\tau]<\infty$, one has
		\begin{equation}
			\mathbb{E}[h(\mathbb{I}_{\{\tau<\sigma_z^{\lambda}\}},\xi_{\tau}^{\lambda,z})]-h(1,0)=\mathbb{E} \left[\displaystyle\int_{0}^{\tau} A^{\lambda,z} h(\mathbb{I}_{\{s<\sigma_z^{\lambda}\}},\xi_s^{\lambda,z}) \mathrm{d}s\right], \quad h \in \mathcal{D}(A^{\lambda,z}).
		\end{equation}
		In particular, we have
		\begin{multline}
			h(1,0)+\mathbb{E}\left[\displaystyle\int_0^{t} \left(-\lambda\,\dfrac{\partial h}{\partial y_2}(1,\xi_s^{\lambda,z}) +  \dfrac{1}{2}\,\dfrac{\partial^2 h}{\partial ^2y_2}(1,\xi_s^{\lambda,z})\right) \mathbb{I}_{\{z<\xi_s^{\lambda,z}\}} \mathrm{d} s\right]\\
			+\mathbb{E}\left[\displaystyle\int_0^{t} \left(\lambda\,\dfrac{\partial h}{\partial y_2}(1,\xi_s^{\lambda,z}) +  \dfrac{1}{2}\,\dfrac{\partial^2 h}{\partial ^2y_2}(1,\xi_s^{\lambda,z})\right) \mathbb{I}_{\{\xi_s^{\lambda,z}<z\}} \mathrm{d} s\right]=0
		\end{multline}
		for every $h \in \mathcal{D}(A^{\lambda,z})$.
	\end{remark}
	\begin{example}
		The process $N^{\lambda,z}=(N_t^{\lambda,z},\,t\geq 0)$ given by
		\begin{multline*}
			N_t^{\lambda,z}=(\xi_t^{\lambda,z}-z)^2\exp(-(\xi_t^{\lambda,z}-z)^2)-z^2\exp(-z^2)\\-\displaystyle\int_0^{t}\left(2\lambda (\xi_s^{\lambda,z}-z)((\xi_s^{\lambda,z}-z)^2-1)+[1-5(\xi_s^{\lambda,z}-z)^2+2(\xi_s^{\lambda,z}-z)^4]\right)\exp(-(\xi_s^{\lambda,z}-z)^2)\mathbb{I}_{\{z<\xi_s^{\lambda,z}\}} \mathrm{d} s\\
			-\displaystyle\int_0^{t}\left(2\lambda (z-\xi_s^{\lambda,z})((\xi_s^{\lambda,z}-z)^2-1)+[1-5(\xi_s^{\lambda,z}-z)^2+2(\xi_s^{\lambda,z}-z)^4]\right)\exp(-(\xi_s^{\lambda,z}-z)^2)\mathbb{I}_{\{\xi_s^{\lambda,z}<z\}} \mathrm{d} s
		\end{multline*}
		is an $\mathbb{F}^{\xi^{\lambda,z}}$-martingale.
	\end{example}
	\begin{remark}
		Let $h\in \mathcal{D}(A^{\lambda,z})$, the function $(t,y_1,y_2)\longrightarrow Q_th(y_1,y_2)$ is the solution to the following partial differential equation
		\begin{align*}
			\begin{array}{ll}
				\begin{cases} \dfrac{\partial u}{\partial t}(t,y_1,y_2)=0 & y_1=0, y_2\in\mathbb{R} \\
					\dfrac{\partial u}{\partial t}(t,y_1,y_2)=-\lambda y_1\,\dfrac{\partial u}{\partial y_2}(t,y_1,y_2) +  \dfrac{y_1}{2}\,\dfrac{\partial^2 u}{\partial ^2y_2}(t,y_1,y_2) &  y_1=1, z<y_2 \\
					\dfrac{\partial u}{\partial t}(t,y_1,y_2)=\dfrac{y_1}{2}\,\dfrac{\partial^2 u}{\partial ^2y_2}(t,y_1,y_2) &  y_1=1, y_2=z \\
					\dfrac{\partial u}{\partial t}(t,y_1,y_2)=\lambda y_1\,\dfrac{\partial u}{\partial y_2}(t,y_1,y_2) +  \dfrac{y_1}{2}\,\dfrac{\partial^2 u}{\partial ^2y_2}(t,y_1,y_2) &  y_1=1, z>y_2\\
					u(0,y_1,y_2)=h(y_1,y_2)& (y_1,y_2)\in \mathcal{S}.
				\end{cases}
			\end{array}
		\end{align*}
	\end{remark}
	%
	%
	\appendix
	\section{Foundational Results for the Main Theorems}
	\label{Appendix_Foundation}
	This section contains several preliminary results that are essential for establishing the main findings of this paper.
	\begin{lemma}
		Let $t > 0$, $g : (0,+\infty) \longrightarrow \mathbb{R}$ be a Borel function such that
		$\mathbb{E}[g(\sigma_z^{\lambda})]<\infty$. Then, $\mathbb{P}$-a.s., 
		\begin{multline}
			\mathbb{E}[g(\sigma_z^{\lambda})\vert\xi_t^{\lambda,z}]=\lambda\,A^{\lambda}(t,z)\displaystyle\int_0^{t}g(r)p(r,z,\lambda\,r)\mathrm{d}r\,\mathbb{I}_{\{\xi_t^{\lambda,z}=z \}}\\
			+\lambda\,\exp(\lambda\alpha(z,\xi_t^{\lambda,z}))\displaystyle\int_t^{\infty}g(r)p(r-t,z-\xi_t^{\lambda,z},\lambda(r-t))\mathrm{d}r\,\mathbb{I}_{\{\xi_t^{\lambda,z}\neq z \}},\label{eqlawofsigmigivenxit}
		\end{multline}	
		where 
		$$A^{\lambda}(t,z)=[F_{\sigma_z^{\lambda}}(t)]^{-1}=\bigg[\Phi\bigg(\lambda\sqrt{t}-\dfrac{z}{\sqrt{t}}\bigg)-\exp(2\lambda\,z)\,\Phi\bigg(-\lambda\sqrt{t}-\dfrac{z}{\sqrt{t}}\bigg)\bigg]^{-1}.$$
	\end{lemma}
	\begin{proof}
		From \eqref{eqpredictableformulla}, we have for any $B\in \mathcal{B}(\mathbb{R})$
		\begin{multline*}
			\mathbb{P}(\xi_t^{\lambda,z}\in B\vert\sigma_z^{\lambda}=r)=\displaystyle\int_{B}\bigg(\mathbb{I}_{\{x=z\}}\mathbb{I}_{\{r\leq t \}}+\dfrac{p(r-t,z-x,\lambda(r-t))p(t,x,\lambda\,t)}{p(r,z,\lambda\,r)}\mathbb{I}_{\{x\neq z\}}\mathbb{I}_{\{t<r \}}\bigg)\,\varrho(\mathrm{d}x).
		\end{multline*}
		From Bayes formula (see  \cite{Shi}  p. 272) it may be concluded
		\begin{align*}
			\mathbb{E}[g(\sigma_z^{\lambda})\vert\xi_t^{\lambda,z}]
			=\lambda\,\dfrac{\displaystyle\int_0^tg(r)p(r,z,\lambda\,r)\mathrm{d}r}{F_{\sigma_z^{\lambda}}(t)}\mathbb{I}_{\{\xi_t^{\lambda,z}=z\}}+\dfrac{\displaystyle\int_t^{+\infty}g(r)p(r-t,z-\xi_t^{\lambda,z},\lambda\,(r-t))\mathrm{d}r}{\displaystyle\int_t^{+\infty}p(r-t,z-\xi_t^{\lambda,z},\lambda\,(r-t))\mathrm{d}r}\mathbb{I}_{\{\xi_t^{\lambda,z}\neq z\}}.
		\end{align*}
		From \eqref{eqintformullanum} and \eqref{eqintformulladen}, it follows that 
		$$ F_{\sigma_z^{\lambda}}(t)=\Phi\bigg(\lambda\sqrt{t}-\dfrac{z}{\sqrt{t}}\bigg)-\exp(2\lambda\,z)\,\Phi\bigg(-\lambda\sqrt{t}-\dfrac{z}{\sqrt{t}}\bigg) $$
		and that
		$$ \displaystyle\int_t^{+\infty}p(r-t,z-\xi_t^{\lambda,z},\lambda\,(r-t))\mathrm{d}r= \dfrac{\exp(-\lambda\alpha(z,\xi_t^{\lambda,z}))}{\lambda}.$$
		Which completes the proof.
	\end{proof}
	\begin{lemma}
		Let $0<t<u$, and $h : (0,+\infty)\times\mathbb{R} \longrightarrow \mathbb{R}$ be a Borel function such that
		$\mathbb{E}[h(\sigma_z^{\lambda},\xi_u^{\lambda,z})]<\infty$. Then, $\mathbb{P}$-a.s., 	
		\begin{multline}
			\mathbb{E}[h(\sigma_z^{\lambda},\xi_u^{\lambda,z})\vert\xi_t^{\lambda,z}]=\lambda\,A^{\lambda}(t,z)\displaystyle\int_0^{t}h(r,z)p(r,z,\lambda\,r)\mathrm{d}r\,\mathbb{I}_{\{\xi_t^{\lambda,z}=z \}}\nonumber\\
			+\lambda\,\exp(\lambda\alpha(z,\xi_t^{\lambda,z}))\bigg[\displaystyle\int_t^{u}h(r,z)p(r-t,z-\xi_t^{\lambda,z},\lambda(r-t))\mathrm{d}r+\\
			\displaystyle\int_{\mathbb{R}}\displaystyle\int_{u}^{+\infty}h(r,y)p(r-u,z-y,\lambda\,(r-u))\mathrm{d}r\,p(u-t,y-\xi_t^{\lambda,z},\lambda\,(u-t))\mathrm{d}y\bigg]\mathbb{I}_{\{\xi_t^{\lambda,z}\neq z \}}
		\end{multline}
	\end{lemma}
	\begin{proof}
		First, we have
		\begin{multline*}
			\mathbb{E}[h(\sigma_z^{\lambda},\xi_u^{\lambda,z})\vert\xi_t^{\lambda,z}]=\mathbb{E}[h(\sigma_z^{\lambda},z)\mathbb{I}_{\{\sigma_z^{\lambda}\leq t\}}\vert\xi_t^{\lambda,z}]+\mathbb{E}[h(\sigma_z^{\lambda},z)\mathbb{I}_{\{t<\sigma_z^{\lambda}\leq u\}}\vert\xi_t^{\lambda,z}]+\mathbb{E}[h(\sigma_z^{\lambda},\xi_u^{\lambda,z})\mathbb{I}_{\{u<\sigma_z^{\lambda}\}}\vert\xi_t^{\lambda,z}]
		\end{multline*}
		The first and second expectations can be computed using the previous lemma. For the third expectation, we proceed as follows:
		\begin{align*}
			&\mathbb{E}[h(\sigma_z^{\lambda},\xi_u^{\lambda,z})\mathbb{I}_{\{u<\sigma_z^{\lambda}\}}\vert\xi_t^{\lambda,z}=x]=\displaystyle\int_{u}^{+\infty}\mathbb{E}[h(\sigma_z^{\lambda},\xi_u^{\lambda,z})\mathbb{I}_{\{u<\sigma_z^{\lambda}\}}\vert\xi_t^{\lambda,z}=x,\sigma_z^{\lambda}=r]\mathbb{P}_{\sigma_z^{\lambda}\vert \xi_t^{\lambda,z}=x }(\mathrm{d}r)\\
			&=\displaystyle\int_{u}^{+\infty}\mathbb{E}[h(r,B_u^{\lambda})\vert B_t^{\lambda}=x,B_r^{\lambda}=z]\mathbb{P}_{\sigma_z^{\lambda}\vert \xi_t^{\lambda,z}=x }(\mathrm{d}r)\\
			&=\displaystyle\int_{u}^{+\infty}\displaystyle\int_{\mathbb{R}}h(r,y)\dfrac{p(r-u,z-y,\lambda\,(r-u))p(u-t,y-x,\lambda\,(u-t))}{p(r-t,z-x,\lambda\,(r-t))}\mathrm{d}y\,\mathbb{P}_{\sigma_z^{\lambda}\vert \xi_t^{\lambda,z}=x }(\mathrm{d}r)\\
			&=\lambda\,\exp(\lambda\alpha(z,x))\\
			&\times\displaystyle\int_{\mathbb{R}}\displaystyle\int_{u}^{+\infty}h(r,y)p(r-u,z-y,\lambda\,(r-u))\mathrm{d}r\,p(u-t,y-x,\lambda\,(u-t))\mathrm{d}y\,\mathbb{I}_{\{x\neq z\}}
		\end{align*}
		This completes the proof.
	\end{proof}
	\begin{remark}
		It follows from \eqref{eqintformullanum}, \eqref{eqintformulladen} and \eqref{eqmodification} that, $\mathbb{P}$-a.s.,
		\begin{multline}
			\mathbb{E}[f(\xi_u^{\lambda,z})\vert\xi_t^{\lambda,z}]=f(z)\,\mathbb{I}_{\{\sigma_z^{\lambda}\leq t \}}+f(z)\\
			\times\bigg[ \Phi\bigg(\lambda\sqrt{u-t}- \dfrac{\vert z-\xi_t^{\lambda,z} \vert}{\sqrt{u-t}}\bigg)-\exp(2\lambda\,\vert z-\xi_t^{\lambda,z} \vert)\Phi\bigg(-\lambda\sqrt{u-t}- \dfrac{\vert z-\xi_t^{\lambda,z} \vert}{\sqrt{u-t}}\bigg)\bigg]\mathbb{I}_{\{t<\sigma_z^{\lambda} \}}\\
			+\exp(\lambda\alpha(z,\xi_t^{\lambda,z}))\displaystyle\int_{\mathbb{R}} f(y)\exp(-\lambda\alpha(z,y))p(u-t,y-\xi_t^{\lambda,z},\lambda\,(u-t))\mathrm{d}y\,\mathbb{I}_{\{t<\sigma_z^{\lambda}\}}.\label{eqlawofxiugivenxit}
		\end{multline}
	\end{remark}
	\begin{corollary}
		Let $t > 0$, $g : (0,+\infty) \longrightarrow \mathbb{R}$ be a Borel function such that
		$\mathbb{E}[g(\sigma_z^{\lambda})]<\infty$. Then, $\mathbb{P}$-a.s., 
		\begin{multline}
			\mathbb{E}[g(\sigma_z^{\lambda})\vert\mathcal{F}_t^{\xi^{\lambda,z}}]=g(\sigma_z^{\lambda})\,\mathbb{I}_{\{\sigma_z^{\lambda}\leq t \}}\\
			+\lambda\,\exp(\lambda\alpha(z,\xi_t^{\lambda,z}))\displaystyle\int_t^{+\infty}g(r)p(r-t,z-\xi_t^{\lambda,z},\lambda(r-t))\mathrm{d}r\,\mathbb{I}_{\{t<\sigma_z^{\lambda} \}}.
		\end{multline}	
	\end{corollary}
	\begin{remark}
		For any positive reel number $t$, we have
		\begin{align*}
			\mathbb{E}[\sigma_z^{\lambda}\vert\xi_t^{\lambda,z}]=\bigg[\dfrac{1}{\lambda^2}+\dfrac{z}{\lambda}A^{\lambda}(t,z)B^{\lambda}(t,z)-2\dfrac{t}{\lambda}\,p(t,z,\lambda\,t)A^{\lambda}(t,z)\bigg]\mathbb{I}_{\{\sigma_z^{\lambda}\leq t\}}+
			\bigg[\dfrac{1}{\lambda^2}+\dfrac{\vert z-\xi^{\lambda,z}_t \vert}{\lambda}+t\bigg]\mathbb{I}_{\{t<\sigma_z^{\lambda}\}}
		\end{align*}
		where,
		\begin{equation}
			B^{\lambda}(t,z)=\Phi\bigg(\lambda\sqrt{t}-\dfrac{z}{\sqrt{t}}\bigg)+\exp(2\lambda\,z)\,\Phi\bigg(-\lambda\sqrt{t}-\dfrac{z}{\sqrt{t}}\bigg).
		\end{equation}
		The expected value of $\sigma_z^{\lambda}$ with respect to $\mathcal{F}_t^{\xi^{\lambda,z}}$ is given by
		\begin{align}
			\mathbb{E}[\sigma_z^{\lambda}\vert\mathcal{F}_t^{\xi^{\lambda,z}}]&=\sigma_z^{\lambda}\,\mathbb{I}_{\{\sigma_z^{\lambda}\leq t \}}+\bigg[\dfrac{1}{\lambda^2}+\dfrac{\vert z-\xi^{\lambda,z}_t \vert}{\lambda}+t\bigg]\mathbb{I}_{\{t<\sigma_z^{\lambda}\}}.
		\end{align}	
	\end{remark}
	\section{Proof of Theorem \ref{thmMarkov} and Corollary \ref{corfiltrationrightcontinuity}}
	\label{Appendix_section_Proof}
	\subsection{Proof of Theorem \ref{thmMarkov}}
	The Markov property of the process $\xi^{\lambda,z}$ can be demonstrated by observing that, given $\sigma_z^{\lambda}=r$, the process $\xi^{\lambda,z}$ is a Gaussian process with certain independence properties. However, in the following proof, we establish the Markov property by using a change of measure.
	
	Since $\zeta_0=0$, it suffices to prove that for all $n\in \mathbb{N}$, for every bounded measurable function $g$ defined on $\mathbb{R}$, and 
	$0 < t_1 < t_2 <\ldots < t_n < u$, we have $\mathbb{P}$-a.s.,
	\begin{equation}
		\mathbb{E}[g(\xi^{\lambda,z}_u)\vert\xi^{\lambda,z}_{t_1},\ldots,\xi^{\lambda,z}_{t_n}]=\mathbb{E}[g(\xi^{\lambda,z}_u)\vert\xi^{\lambda,z}_{t_n}].
	\end{equation}
	Since for all $t>0$, $\mathbb{I}_{\{t<\sigma_z^{\lambda}\}}=\mathbb{I}_{\{\xi^{\lambda,z}_t=z\}}$, $\mathbb{P}$-a.s. it follows that
	\begin{align*}
		\mathbb{E}[g(\xi^{\lambda,z}_u)\mathbb{I}_{\{\sigma_z^{\lambda}\leq t_n\}}\vert\xi^{\lambda,z}_{t_1},\ldots,\xi^{\lambda,z}_{t_n}]&=\mathbb{E}[g(\xi^{\lambda,z}_{t_n})\mathbb{I}_{\{\sigma_z^{\lambda}\leq t_n\}}\vert\xi^{\lambda,z}_{t_1},\ldots,\xi^{\lambda,z}_{t_n}]\\
		&=\mathbb{E}[g(\xi^{\lambda,z}_u)\mathbb{I}_{\{\sigma_z^{\lambda}\leq t_n\}}\vert\xi^{\lambda,z}_{t_n}].
	\end{align*}
	Hence, it remains to show that
	\begin{equation*}
		\mathbb{E}[g(\xi^{\lambda,z}_u)\mathbb{I}_{\{t_n<\sigma_z^{\lambda}\}}\vert\xi^{\lambda,z}_{t_1},\ldots,\xi^{\lambda,z}_{t_n}]=\mathbb{E}[g(\xi^{\lambda,z}_u)\mathbb{I}_{\{t_n<\sigma_z^{\lambda}\}}\vert\xi^{\lambda,z}_{t_n}].
	\end{equation*}
	Thus, it suffices to verify that for every bounded measurable function $L$ on $\mathbb{R}^n$, we have
	\begin{equation}
		\mathbb{E}[g(\xi^{\lambda,z}_u)\mathbb{I}_{\{t_n<\sigma_z^{\lambda}\}}L(\xi^{\lambda,z}_{t_1},\ldots,\xi^{\lambda,z}_{t_n})]=\mathbb{E}[\mathbb{E}[g(\xi^{\lambda,z}_u)\vert\xi^{\lambda,z}_{t_n}]\mathbb{I}_{\{t_n<\sigma_z^{\lambda}\}}L(\xi^{\lambda,z}_{t_1},\ldots,\xi^{\lambda,z}_{t_n})].\label{eqMart<tau}
	\end{equation}
	We see that
	\begin{align*}
		\mathbb{E}[g(\xi^{\lambda,z}_u)\mathbb{I}_{\{t_n<\sigma_z^{\lambda}\leq u\}}L(\xi^{\lambda,z}_{t_1},&\ldots,\xi^{\lambda,z}_{t_n})]=g(z)\mathbb{E}[\mathbb{I}_{\{t_n<\sigma_z^{\lambda}\leq u\}}L(B^{\lambda}_{t_1},\ldots,B^{\lambda}_{t_n})]\\
		&=g(z)\displaystyle\int_{t_n}^{u}\mathbb{E}[L(B^{\lambda}_{t_1},\ldots,B^{\lambda}_{t_n})\vert \sigma_z^{\lambda}=r]\mathbb{P}_{\sigma_z^{\lambda}}(\mathrm{d}r)\\
		&=g(z)\displaystyle\int_{t_n}^{u}\mathbb{E}[L(B^{\lambda}_{t_1},\ldots,B^{\lambda}_{t_n})\vert B^{\lambda}_{r}=z]\mathbb{P}_{\sigma_z^{\lambda}}(\mathrm{d}r).
	\end{align*}
	With the notation,
	\begin{equation}
		U^{\lambda,z}(r,t,x)=\dfrac{p(r-t,z-x,\lambda(r-t))}{p(r,z,\lambda\,r)},\,\,t<r,\,\, x\in\mathbb{R}.
	\end{equation}
	The change of probability, Fubini, and \eqref{eqlawofsigmigivenxit} yield
	\begin{align*}
		&\mathbb{E}[g(\xi^{\lambda,z}_u)\mathbb{I}_{\{t_n<\sigma_z^{\lambda}\leq u\}}L(\xi^{\lambda,z}_{t_1},\ldots,\xi^{\lambda,z}_{t_n})]\\
		&=g(z)\displaystyle\int_{t_n}^{u}\mathbb{E}\left[L(B^{\lambda}_{t_1},\ldots,B^{\lambda}_{t_n})\dfrac{p(r-t_n,z-B^{\lambda}_{t_n},\lambda(r-t_n))}{p(r,z,\lambda\,r)}\right]\mathbb{P}_{\sigma_z^{\lambda}}(\mathrm{d}r)\nonumber\\
		&=\lambda\, g(z)\mathbb{E}\left[L(B^{\lambda}_{t_1},\ldots,B^{\lambda}_{t_n})\displaystyle\int_{t_n}^{u}p(r-t_n,z-B^{\lambda}_{t_n},\lambda(r-t_n))\mathrm{d}r\right]\\
		&=\displaystyle\int_{t_n}^{+\infty} \mathbb{E}\left[L(B^{\lambda}_{t_1},\ldots,B^{\lambda}_{t_n})\mathbb{E}[g(z)\mathbb{I}_{\{t_n<\sigma_z^{\lambda}\leq u\}}\vert\xi^{\lambda,z}_{t_n}=x ]_{x=B^{\lambda}_{t_n}}U^{\lambda,z}(r,t_n,B^{\lambda}_{t_n})\right]\mathbb{P}_{\sigma_z^{\lambda}}(\mathrm{d}r)\\
		&=\displaystyle\int_{t_n}^{+\infty} \mathbb{E}\left[L(B^{\lambda}_{t_1},\ldots,B^{\lambda}_{t_n})\mathbb{E}[g(z)\mathbb{I}_{\{t_n<\sigma_z^{\lambda}\leq u\}}\vert\xi^{\lambda,z}_{t_n}=x ]_{x=B^{\lambda}_{t_n}}\vert B^{\lambda}_{r}=z\right]\mathbb{P}_{\sigma_z^{\lambda}}(\mathrm{d}r)\\
		&=\displaystyle\int_{t_n}^{+\infty} \mathbb{E}\left[L(B^{\lambda}_{t_1},\ldots,B^{\lambda}_{t_n})\mathbb{E}[g(z)\mathbb{I}_{\{t_n<\sigma_z^{\lambda}\leq u\}}\vert\xi^{\lambda,z}_{t_n}=x ]_{x=B^{\lambda}_{t_n}}\vert \sigma_z^{\lambda}=r\right]\mathbb{P}_{\sigma_z^{\lambda}}(\mathrm{d}r)\\
		&=\mathbb{E}[\mathbb{E}[g(\xi^{\lambda,z}_u)\mathbb{I}_{\{t_n<\sigma_z^{\lambda}\leq u\}}\vert \xi^{\lambda,z}_{t_n}]L(\xi^{\lambda,z}_{t_1},\ldots,\xi^{\lambda,z}_{t_n})].
	\end{align*}
	Similarly, The change of probability, Fubini, and \eqref{eqlawofxiugivenxit} yield
	\begin{align*}
		\mathbb{E}[g(\xi^{\lambda,z}_u)\mathbb{I}_{\{u<\sigma_z^{\lambda}\}}L(\xi^{\lambda,z}_{t_1},\ldots,\xi^{\lambda,z}_{t_n})]=\mathbb{E}[\mathbb{E}[g(\xi^{\lambda,z}_u)\mathbb{I}_{\{u<\sigma_z^{\lambda}\}}\vert \xi^{\lambda,z}_{t_n}]L(\xi^{\lambda,z}_{t_1},\ldots,\xi^{\lambda,z}_{t_n})].
	\end{align*}
	\subsection{Proof of Corollary \ref{corfiltrationrightcontinuity}}
	We recall that if $X$ is an $\mathbb{F}^{X}_+$-Markov process, then $\mathbb{F}^{X}$ is right-continuous (see, \cite[Ch. I, (8.12)]{BG}). Let $f$ be a bounded continuous function. For all $0\leq t<u$. We have, $\mathbb{P}$-a.s.,
	\begin{equation*}
		\mathbb{E}\left[f(\xi_u^{\lambda,z})\Big\vert \mathcal{F}^{\xi^{\lambda,z}}_{t+}\right]=\mathbb{E}\left[f(\xi_u^{\lambda,z})\bigg\vert \bigcap\limits_{n}\mathcal{F}^{\xi^{\lambda,z}}_{t_n}\right]
		=\lim\limits_{n\rightarrow +\infty}\mathbb{E}\left[f(\xi_u^{\lambda,z})\Big\vert \xi^{\sigma_z^{\lambda}}_{t_n}\right]
	\end{equation*}
	where, $(t_n)_{n\in \mathbb{N}}$ is a decreasing sequence of strictly positive real numbers converging to $t$: that is $0 \leq t <...< t_{n+1} < t_{n}<...<t_{1} < u $, $t_{n} \searrow t$ as $n \mapsto +\infty $. It follows from \eqref{eqmodification} and \eqref{eqsemigroup} that, $\mathbb{P}$-a.s.,
	\begin{align*}
		\mathbb{E}[f(\xi_u^{\lambda,z})\vert\xi_{t_n}^{\lambda,z}]=f(z)\,\mathbb{I}_{\{\sigma_z^{\lambda}\leq t_n \}}&+f(z)\bigg[ \Phi\bigg(\lambda\sqrt{u-t_n}- \dfrac{\vert z-\xi_{t_n}^{\lambda,z} \vert}{\sqrt{u-t_n}}\bigg)\\-\exp(2\lambda\,\vert z-\xi_{t_n}^{\lambda,z} \vert)&\Phi\bigg(-\lambda\sqrt{u-t_n}- \dfrac{\vert z-\xi_{t_n}^{\lambda,z} \vert}{\sqrt{u-t_n}}\bigg)\bigg]\mathbb{I}_{\{t_n<\sigma_z^{\lambda} \}}\nonumber\\
		+\exp(\lambda\alpha(z,\xi_{t_n}^{\lambda,z}))\displaystyle\int_{\mathbb{R}} f(y)&\exp(-\lambda\alpha(z,y))p(u-t_n,y-\xi_{t_n}^{\lambda,z},\lambda\,(u-t_n))\mathrm{d}y\,\mathbb{I}_{\{t_n<\sigma_z^{\lambda}\}}.
	\end{align*}
	Using the fact that $y\longrightarrow f(y)\exp(-\lambda\,\alpha(z,y) )$ is bounded and continuous, the weak convergence of the Gaussian measures, the continuity of $\Phi$ and the process $\xi^{\sigma_z^{\lambda}}$ and the right continuity of the indicator functions in the above equation, we obtain, $\mathbb{P}$-a.s.,
	\begin{equation*}
		\mathbb{E}\left[f(\xi_u^{\lambda,z})\Big\vert \mathcal{F}^{\xi^{\lambda,z}}_{t+}\right]=\mathbb{E}\left[f(\xi_u^{\lambda,z})\Big\vert \xi^{\lambda,z}_{t}\right].
	\end{equation*}
	Hence, $\xi^{\lambda,z}$ is an $\mathbb{F}^{\xi^{\lambda,z}}_+$-Markov process. Thus, we have the desired result.
	\section{Integral Formulas}
	\label{Appendix_integrals}
	In this section, we present the detailed computations of the integrals used in the proofs of the main results. By applying the following known integral (cf. \cite{AS}, p. 304): For $a$, $b\in\mathbb{R}$
	\begin{multline}
		\displaystyle\int \exp\bigg( -a^2r^2-\dfrac{b^2}{r^2} \bigg)\mathrm{d}r=\dfrac{\sqrt{\pi}}{2\sqrt{a^2}}\bigg[\exp(2\sqrt{a^2}\sqrt{b^2})\,\Phi\bigg(\sqrt{2a^2}\,r+\dfrac{\sqrt{2b^2}}{r} \bigg) \\
		+\exp(-2\sqrt{a^2}\sqrt{b^2})\,\Phi\bigg(\sqrt{2a^2}\,r-\dfrac{\sqrt{2b^2}}{r}  \bigg)\bigg]+c,\label{eqprimitive}
	\end{multline}
	where $c$ is a real constant, we derive the following integrals. For any $0\leq s<t$, we have:
	\begin{align}
		\displaystyle\int_{s}^{t}p(r-s,z&-\xi^{\lambda,z}_s,\lambda(r-s))\mathrm{d}r
		=\sqrt{\dfrac{2}{\pi}}\exp(\lambda(z-\xi^{\lambda,z}_s)\displaystyle\int_{0}^{\sqrt{t-s}} \exp\bigg(-\dfrac{1}{2}\lambda^2r^2-\dfrac{1}{2}\dfrac{(z-\xi^{\lambda,z}_s)^2}{r^2}\bigg)\mathrm{d}r\nonumber\\
		&=\dfrac{1}{\lambda}\exp(\lambda(z-\xi^{\lambda,z}_s))\bigg[ \exp(-\lambda\vert z-\xi^{\lambda,z}_s\vert)\Phi\bigg(\lambda\sqrt{t-s}- \dfrac{\vert z-\xi^{\lambda,z}_s \vert}{\sqrt{t-s}}\bigg)\nonumber\\
		&-\exp(\lambda\vert z-\xi^{\lambda,z}_s\vert)\Phi\bigg(-\lambda\sqrt{t-s}- \dfrac{\vert z-\xi^{\lambda,z}_s \vert}{\sqrt{t-s}}\bigg) \bigg]\nonumber\\
		&=\dfrac{1}{\lambda}\bigg[ \exp(-\lambda\alpha(z,\xi^{\lambda,z}_s))\Phi\bigg(\lambda\sqrt{t-s}- \dfrac{\vert z-\xi^{\lambda,z}_s \vert}{\sqrt{t-s}}\bigg)\nonumber\\
		&-\exp(\lambda\gamma( z,\xi^{\lambda,z}_s))\Phi\bigg(-\lambda\sqrt{t-s}- \dfrac{\vert z-\xi^{\lambda,z}_s \vert}{\sqrt{t-s}}\bigg) \bigg],\label{eqintformullanum}
	\end{align}
	and
	\begin{align}
		\displaystyle\int_{s}^{+\infty}p(r-s,z-\xi^{\lambda,z}_s,\lambda(r-s))\mathrm{d}r&=\sqrt{\dfrac{2}{\pi}}\exp(\lambda(z-\xi^{\lambda,z}_s)\displaystyle\int_{0}^{+\infty} \exp\bigg(-\dfrac{1}{2}\lambda^2r^2-\dfrac{1}{2}\dfrac{(z-\xi^{\lambda,z}_s)^2}{r^2}\bigg)\mathrm{d}r\nonumber\\
		&=\dfrac{1}{\lambda}\exp(\lambda(z-\xi^{\lambda,z}_s))\exp(-\lambda\vert z-\xi^{\lambda,z}_s\vert)\nonumber\\
		&=\dfrac{1}{\lambda}\exp(-\lambda \alpha(z,\xi^{\lambda,z}_s)).\label{eqintformulladen}
	\end{align}
	Using the fact that the primitive function of $x^2\exp\bigg( -\dfrac{1}{2}a^2x^2-\dfrac{1}{2}\dfrac{b^2}{x^2} \bigg)$, for $a> 0$ and $b \in \mathbb{R}$, is
	\begin{align*}
		\displaystyle\int x^2\,\exp\bigg( -\dfrac{1}{2}a^2x^2-\dfrac{1}{2}\dfrac{b^2}{x^2} \bigg)\mathrm{d}x&=\sqrt{\dfrac{\pi}{2}}\dfrac{1}{a^3}\bigg[\exp(\sqrt{a^2}\sqrt{b^2})\,\Phi\bigg(\sqrt{a^2}\,x+\dfrac{\sqrt{b^2}}{x}  \bigg)\nonumber\\
		&-\exp(-\sqrt{a^2}\sqrt{b^2})\,\Phi\bigg(-\sqrt{a^2}\,x+\dfrac{\sqrt{b^2}}{x}  \bigg)\bigg]\\
		&-\sqrt{\dfrac{\pi}{2}}\dfrac{\sqrt{b^2}}{a^2}\bigg[\exp(\sqrt{a^2}\sqrt{b^2})\,\Phi\bigg(\sqrt{a^2}\,x+\dfrac{\sqrt{b^2}}{x}  \bigg)\\
		&+\exp(-\sqrt{a^2}\sqrt{b^2})\,\Phi\bigg(-\sqrt{a^2}\,x+\dfrac{\sqrt{b^2}}{x}  \bigg)\bigg]\\
		&-\dfrac{x}{a^2}\exp\bigg( -\dfrac{1}{2}a^2x^2-\dfrac{1}{2}\dfrac{b^2}{x^2} \bigg)+c,
	\end{align*}
	$c\in \mathbb{R}$, we derive the following formula
	\begin{align}
		\displaystyle\int_s^{t}(r&-s)\,p(r-s,z-\xi^{\lambda,z}_s,\lambda(r-s))\mathrm{d}r\nonumber\\
		&=\sqrt{\dfrac{2}{\pi}}\exp(\lambda(z-\xi^{\lambda,z}_s)\displaystyle\int_{0}^{\sqrt{t-s}}r^2 \exp\bigg(-\dfrac{1}{2}\lambda^2r^2-\dfrac{1}{2}\dfrac{(z-\xi^{\lambda,z}_s)^2}{r^2}\bigg)\mathrm{d}r\nonumber\\
		&=\dfrac{1}{\lambda^3}\exp(-\lambda\alpha(z,\xi^{\lambda,z}_s))\Phi\bigg(\lambda\sqrt{t-s}- \dfrac{\vert z-\xi^{\lambda,z}_s \vert}{\sqrt{t-s}}\bigg)\nonumber\\
		&-\dfrac{1}{\lambda^3}\exp(\lambda\gamma(z,\xi^{\lambda,z}_s))\Phi\bigg(-\lambda\sqrt{t-s}- \dfrac{\vert z-\xi^{\lambda,z}_s \vert}{\sqrt{t-s}}\bigg)\nonumber\\
		&+\dfrac{\vert z-\xi^{\lambda,z}_s \vert}{\lambda^2}\exp(-\lambda\alpha(z,\xi^{\lambda,z}_s))\Phi\bigg(\lambda\sqrt{t-s}- \dfrac{\vert z-\xi^{\lambda,z}_s \vert}{\sqrt{t-s}}\bigg)\nonumber\\
		&+\dfrac{\vert z-\xi^{\lambda,z}_s \vert}{\lambda^2}\exp(\lambda\gamma(z,\xi^{\lambda,z}_s))\Phi\bigg(-\lambda\sqrt{t-s}- \dfrac{\vert z-\xi^{\lambda,z}_s \vert}{\sqrt{t-s}}\bigg)\nonumber\\
		&-\dfrac{2}{\lambda^2}(t-s)p(t-s,z-\xi^{\lambda,z}_s,\lambda\,(t-s)).\label{eqintformullanum(r-s)}
	\end{align}
	Using the fact that the primitive function of $\Phi\left(a \sqrt{x}+\frac{b}{\sqrt{x}}\right)$, for $a, b \in \mathbb{R}$, is
	\begin{align*}
		\int \Phi\left(a \sqrt{x}+\frac{b}{\sqrt{x}}\right)& \mathrm{d} x=\frac{1}{2 a \sqrt{a^2}} e^{-\sqrt{a^2} \sqrt{b^2}-a b}
		\left(-\sqrt{a^2} \sqrt{b^2}+a b-1\right) \Phi\left(\sqrt{a^2} \sqrt{x}-\frac{\sqrt{b^2}}{\sqrt{x}}\right)\\
		&+\frac{1}{2 a \sqrt{a^2}} e^{\sqrt{a^2} \sqrt{b^2}-a b}\left(\sqrt{a^2} \sqrt{b^2}+a b-1\right)\left( \Phi\left(\sqrt{a^2} \sqrt{x}+\frac{\sqrt{b^2}}{\sqrt{x}}\right)-1\right)\\
		&+\frac{\sqrt{b^2}}{2 a} e^{-\sqrt{a^2} \sqrt{b^2}-a b}+x \Phi\left(a \sqrt{x}+\frac{b}{\sqrt{x}}\right)+\frac{ \sqrt{x} e^{-\frac{1}{2}(a \sqrt{x}+\frac{b}{\sqrt{x}})^2}}{\sqrt{2\pi}a}+c,
	\end{align*}
	$c\in \mathbb{R}$, we obtain that, for all $s, t>0$
	\begin{align}
		\displaystyle\int_{0}^{t}\Phi\bigg(\lambda\sqrt{u}&+ \dfrac{ z-\xi^{\lambda,z}_{s} }{\sqrt{u}}\bigg)\mathrm{d}u
		=\frac{1}{2 \lambda^2} \exp(-\lambda\gamma(z,\xi^{\lambda,z}_s))
		\left(-\lambda\alpha(z,\xi^{\lambda,z}_s)-1\right) \Phi\left(\lambda \sqrt{t}-\frac{\vert z-\xi^{\lambda,z}_s\vert}{\sqrt{t}}\right)\nonumber\\
		&+\frac{1}{2 \lambda^2} \exp(\lambda\alpha(z,\xi^{\lambda,z}_s))\left(\lambda\gamma(z,\xi^{\lambda,z}_s)-1\right)
		\left( \Phi\left(\lambda \sqrt{t}+\frac{\vert z-\xi^{\lambda,z}_s\vert}{\sqrt{t}}\right)-1\right) \nonumber\\
		&+t\, \Phi\left(\lambda \sqrt{t}+\frac{ z-\xi^{\lambda,z}_s}{\sqrt{t}}\right)
		+\dfrac{\sqrt{t}}{\lambda}\dfrac{1}{\sqrt{2\pi}}\exp\bigg(-\frac{1}{2}\Big(\lambda \sqrt{t}+\dfrac{z-\xi^{\lambda,z}_s}{\sqrt{t}}\Big)^2\bigg),\label{eqappendixint1}
	\end{align}
	and that
	\begin{align}
		\displaystyle\int_{0}^{t}\Phi\bigg(\lambda\sqrt{u}&- \dfrac{ z-\xi^{\lambda,z}_{s} }{\sqrt{u}}\bigg)\mathrm{d}u=\frac{1}{2 \lambda^2} \exp(-\lambda\alpha(z,\xi^{\lambda,z}_s))
		\left(-\lambda\gamma(z,\xi^{\lambda,z}_s)-1\right) \Phi\left(\lambda \sqrt{t}-\frac{\vert z-\xi^{\lambda,z}_s\vert}{\sqrt{t}}\right)\nonumber\\
		&+\frac{1}{2 \lambda^2} \exp(\lambda\gamma(z,\xi^{\lambda,z}_s))\left(\lambda\alpha(z,\xi^{\lambda,z}_s)-1\right)
		\left( \Phi\left(\lambda \sqrt{t}+\frac{\vert z-\xi^{\lambda,z}_s\vert}{\sqrt{t}}\right)-1\right)\nonumber \\
		&+t\, \Phi\left(\lambda \sqrt{t}-\frac{ z-\xi^{\lambda,z}_s}{\sqrt{t}}\right)
		+\dfrac{\sqrt{t}}{\lambda}\dfrac{1}{\sqrt{2\pi}}\exp\bigg(-\frac{1}{2}\Big(\lambda \sqrt{t}-\dfrac{z-\xi^{\lambda,z}_s}{\sqrt{t}}\Big)^2\bigg).\label{eqappendixint2}
	\end{align}
	\section{Predictable Times and Processes}
	\label{Appendix_Kallenberg}
	In this section, for the sake of easy reference, we provide a brief overview of the results from \cite{K} that characterize stopping times by utilizing Martingales, Compensators, and Feller processes. These results serve as a key tool for classifying stopping times based on the properties of the associated stochastic processes.
	\begin{theorem}
		For any filtration $\mathcal{F}$, these conditions are equivalent:
		\begin{enumerate}
			\item[(i)] Every accessible time is predictable;
			\item[(ii)] $\mathcal{F}_{\tau-}=\mathcal{F}_\tau$ on $\{\tau<\infty\}$ for every predictable time $\tau$;
			\item[(iii)] $\Delta M_\tau=0$ a.s. on $\{\tau<\infty\}$ for every martingale $M$ and predictable time $\tau$.
		\end{enumerate} 
	\end{theorem}
	\begin{proof}
		See, e.g., \cite[Proposition 25.19]{K}
	\end{proof}
	\begin{theorem}\label{thmappendixpredictable}
		Let $X$ be a Feller process with arbitrary initial distribution, and fix any optional time $\tau$. Then these conditions are equivalent:
		\begin{enumerate}
			\item[(i)] $\tau$ is predictable;
			\item[(ii)] $\tau$ is accessible;
			\item[(iii)]$X_{\tau-}=X_\tau$ a.s. on $\{\tau<\infty\}$.
		\end{enumerate} 
	\end{theorem}
	\begin{proof}
		See, e.g., \cite[Proposition 25.20]{K}
	\end{proof}
	\begin{theorem}\label{thmappendixcompensator}
		Let $T$ be a stopping time with compensator $A$. Then:
		\begin{enumerate}
			\item[(i)] $T$ is predictable if and only if $A$ is almost surely constant apart from a possible unit jump;
			\item[(ii)] $T$ is accessible if and only if $A$ is almost surely purely discontinuous;
			\item[(iii)] $T$ is totally inaccessible if and only if $A$ is almost surely continuous.
		\end{enumerate}
	\end{theorem}
	\begin{proof}
		See, e.g., \cite[Corollary 25.18]{K}
	\end{proof}

\end{document}